\documentclass[12pt]{amsart}

\usepackage{amsthm, amsmath, amsfonts, amssymb, graphicx, tikz-cd, enumerate, comment, array, hyperref, calrsfs}

\newtheorem{theorem}{Theorem}[section]
\newtheorem{lemma}[theorem]{Lemma}
\newtheorem{proposition}[theorem]{Proposition}
\newtheorem{corollary}[theorem]{Corollary}

\theoremstyle{definition}
\newtheorem{definition}[theorem]{Definition}

\newtheorem{remark}[theorem]{Remark}

\newcommand{\func}[1]{\operatorname{#1}}

\def\bal{\boldsymbol{\mathit{ba}\ell}}
\def\ubal{\boldsymbol{\mathit{uba}\ell}}

\def\cubal{\boldsymbol{\mathit{cuba}\ell}}
\def\mbal{\boldsymbol{\mathit{mba}\ell}}
\def\mubal{\boldsymbol{\mathit{muba}\ell}}
\def\mcubal{\boldsymbol{\mathit{mcuba}\ell}}

\def\KHaus{{\sf{KHaus}}}
\def\Stone{{\sf{Stone}}}
\def\KHF{\sf{KHF}}
\def\DFr{{\sf{DF}}}

\def\ba{{\sf{BA}}}
\def\ma{{\sf{MA}}}

\def\Alg{{\sf{Alg}}}
\def\Algu{{\sf{Alg}}^u}
\def\Coalg{{\sf{Coalg}}}

\def\FunctorFont{\mathcal}
\def\H{{\FunctorFont{H}}}
\def\Hu{{\FunctorFont{H}}^u}
\def\Hc{{\FunctorFont{H}}^c}
\def\V{{\FunctorFont{V}}}
\def\A{{\FunctorFont{A}}}
\def\B{{\FunctorFont{B}}}
\def\C{{\FunctorFont{C}}}
\def\Y{{\FunctorFont{Y}}}
\def\F{{\FunctorFont{F}}}
\def\G{{\FunctorFont{G}}}
\def\M{{\FunctorFont{M}}}
\def\N{{\FunctorFont{N}}}
\def\SS{{\FunctorFont{S}}}
\def\T{{\FunctorFont{T}}}

\newcommand{\meet}{\land}

\title[The Vietoris functor and modal operators on $C(X)$]{The Vietoris functor and modal operators on rings of continuous functions}
\author{G.~Bezhanishvili, L.~Carai, P.~J.~Morandi}
\date{}

\address{Department of Mathematical Sciences, New Mexico State University, Las Cruces NM 88003 USA}
\email{guram@nmsu.edu, lcarai@nmsu.edu, pmorandi@nmsu.edu}

\subjclass[2010]{54B20; 54D30; 54C30; 06F25; 13J25; 06E15; 06E25; 03B45}
\keywords{Compact Hausdorff space, continuous relation, coalgebra, bounded archimedean $\ell$-algebra, uniform completeness, Gelfand duality, Stone space, modal algebra, J\'onsson-Tarski duality}

\begin{document}

\begin{abstract}
We introduce an endofunctor $\H$ on the category $\bal$ of bounded archimedean $\ell$-algebras and show that there is a dual adjunction between the category $\Alg(\H)$ of algebras for $\H$ and the category $\Coalg(\V)$ of coalgebras for the Vietoris endofunctor $\V$ on the category of compact Hausdorff spaces. We also introduce an endofunctor $\Hu$ on the reflective subcategory of $\bal$ consisting of uniformly complete objects of $\bal$ and show that Gelfand duality lifts to a dual equivalence between $\Alg(\Hu)$ and $\Coalg(\V)$. On the one hand, this generalizes a result of \cite{Abr88,KKV04} for the category of coalgebras of the Vietoris endofunctor on the category of Stone spaces. On the other hand, it yields an alternate proof of a recent result of \cite{BCM20a}.
\end{abstract}

\maketitle

\section{Introduction} \label{sec: intro}

It is a well-known result in modal logic that the category $\ma$ of modal algebras is dually equivalent to the category $\DFr$ of descriptive frames. This result has its origins in the work of J\'onsson and Tarski \cite{JT51}, which is why it is sometimes referred to as J\'onsson-Tarski duality.  
In its present form it was established by Esakia \cite{Esa74} and Goldblatt \cite{Gol76} (but see also Halmos \cite{Hal56}). 

A descriptive frame is a Stone space (compact Hausdorff zero-dimensional space) $X$ equipped with a binary relation $R$ that is continuous, meaning that the corresponding map $\rho_R : X \to \mathcal V(X)$ into the Vietoris space of $X$, given by
\[
\rho_R(x)=R[x]=\{ y\mid xRy \},
\]
is a well-defined continuous map. In fact, $\DFr$ is isomorphic to the category $\Coalg(\V)$ of coalgebras for the Vietoris endofunctor $\V$ on the category $\Stone$ of Stone spaces. Abramsky~\cite{Abr88} and Kupke, Kurz, and Venema \cite{KKV04} defined the dual endofunctor $\H$ on the category $\ba$ of boolean algebras. They showed that the category $\Alg(\H)$ of algebras for $\H$ is isomorphic to $\ma$, and proved that Stone duality between $\ba$ and $\Stone$ lifts to a dual equivalence between $\Alg(\H)$ and $\Coalg(\V)$. This yields an elegant new proof of the Esakia-Goldblatt theorem that $\ma$ is dually equivalent to $\DFr$.

 Let $\KHaus$ be the category of compact Hausdorff spaces and continuous maps. Then $\Stone$ is a full subcategory of $\KHaus$. There are several generalizations of Stone duality to $\KHaus$. To outline one such generalization, we point out that in Stone duality we work with the boolean algebra of clopens, which correspond to continuous characteristic functions. Since an arbitrary compact Hausdorff space $X$ does not have enough clopens, it is natural to work instead with the ring $C(X)$ of all continuous real-valued functions. This gives rise to the celebrated Gelfand duality between $\KHaus$ and the category of bounded archimedean $\ell$-algebras that in addition are uniformly complete (see Section~2.1 for details). Up to isomorphism, these are exactly the rings $C(X)$ for $X\in\KHaus$.
  
The Vietoris endofunctor $\V:\Stone\to\Stone$ is the restriction of the Vietoris endofunctor $\V:\KHaus\to\KHaus$. We call a relation $R$ on $X\in\KHaus$ continuous if $\rho_R:X\to\V(X)$ is a well-defined continuous map. The pairs $(X,R)$, where $X\in\KHaus$ and $R$ is a continuous relation on $R$, generalize descriptive frames. Following \cite{BCM20a}, we call such pairs {\em compact Hausdorff frames}, and denote the resulting category by $\KHF$. Then $\DFr$ is a full subcategory of $\KHF$ and the isomorphism between $\DFr$ and the category of coalgebras for $\V:\Stone\to\Stone$ extends to an isomorphism between $\KHF$ and the category of coalgebras for $\V:\KHaus\to\KHaus$. 

Our aim is to generalize the endofunctor $\H:\ba\to\ba$ that is the algebraic counterpart of $\V:\Stone\to\Stone$ to an endofunctor on the category $\bal$ of bounded archimedean $\ell$-algebras so that it is the algebraic counterpart of $\V:\KHaus\to\KHaus$. For this we need to overcome several obstacles. Firstly, the construction of $\H:\ba\to\ba$ utilizes the existence of free boolean algebras. However, as was shown in \cite{BCM20c}, free algebras on sets do not exist in $\bal$. Instead we need to work with free algebras on weighted sets (see Section~2.2). Secondly, since $\KHaus$ is dually equivalent to the reflective subcategory $\ubal$ of $\bal$ consisting of uniformly complete objects in $\bal$, additional care is needed when transitioning from $\bal$ to its subcategory $\ubal$. 

We first construct an endofunctor $\H:\bal\to\bal$ and prove that there is a dual adjunction between $\Alg(\H)$ and $\Coalg(\V)$. We next construct an endofunctor $\Hu:\ubal\to\ubal$ and prove that there is a dual equivalence between $\Alg(\Hu)$ and $\Coalg(\V)$. As a result, we obtain a generalization of  the result of Kupke, Kurz, and Venema \cite{KKV04} from $\Stone$ to $\KHaus$. 

In \cite{BCM20a} we generalized the notion of a modal operator on a boolean algebra to that of a modal operator on $A\in\bal$. This has resulted in the category $\mbal$ of modal bounded archimedean $\ell$-algebras and its reflective subcategory $\mubal$ consisting of uniformly complete objects. The main result of \cite{BCM20a} establishes that there is a dual adjunction between $\mbal$ and $\KHF$, which restricts to a dual equivalence between $\mubal$ and $\KHF$. On the one hand this generalizes Gelfand duality, and on the other hand the duality between $\ma$ and $\DFr$. Here we show that $\Alg(\H)$ is isomorphic to $\mbal$ and that $\Alg(\Hu)$ is isomorphic to $\mubal$. From this the main result of \cite{BCM20a} follows, and we arrive at the following diagram, where the hook arrows are embeddings of categories, the two-sided horizontal arrows are dual equivalences, and the vertical arrows are isomorphisms.
\[
\begin{tikzcd}
\Alg(\H) \arrow[r, hookleftarrow] & \Alg(\Hu) \arrow[r, <->] & \Coalg(\V) \arrow[d, <->] \\
 \mbal \arrow[r, hookleftarrow] \arrow[u, <->] & \mubal  \arrow[u, <->] \arrow[r, <->] & \KHF
\end{tikzcd}
\]

The paper is organized as follows. Section 2 provides the necessary background for the paper, including Gelfand duality, the construction of free objects in $\bal$, and  the definition of modal operators on objects of $\bal$. In Section 3 we introduce the endofunctor $\H$ on $\bal$, and in Section 4 we show that the category of algebras for $\H$ is isomorphic to $\mbal$. We relate $\H$ to the Vietoris functor $\V$ in Section 5 by showing that for $A \in \bal$, the Yosida space of $\H(A)$ is homeomorphic to the Vietoris space of the Yosida space of $A$. We prove our main result in Section 6, establishing a dual adjunction between $\Alg(\H)$ and $\Coalg(\V)$. We then introduce a reflective subcategory $\Algu(\H)$ of $\Alg(\H)$ and show that this dual adjunction restricts to a dual equivalence between $\Algu(\H)$ and $\Coalg(\V)$. We also provide an alternate view of the category $\Algu(\H)$ as $\Alg(\Hu)$. In Section 7 we derive the main result of \cite{BCM20a}, showing that there is a dual adjunction between $\mbal$ and $\KHF$, which restricts to a dual equivalence between $\mubal$ and $\KHF$. In Section 8 we show how the exclusion of the empty set from the construction of the Vietoris space results in the modification of the $\H$ functor to the functor $\H^*$ such that $\Alg(\H^*)$ is isomorphic to the full subcategory $\mbal^{\sf D}$ of those $(A,\Box)$ where $\Box$ corresponds to a serial relation. 
Finally, in Section 9 we relate our results to those of \cite{KKV04}.

\section{Preliminaries} \label{sec: prelim}

In this section we provide the necessary background for the rest of the paper. In \S~\ref{subsec: Gelfand duality} we recall Gelfand duality, in \S~\ref{subsec: free objects} free objects in $\bal$ over weighted sets, and finally in \S~\ref{subsec: modal operators} modal operators on algebras in $\bal$ and a generalization of Gelfand duality to this setting.

\subsection{Gelfand duality} \label{subsec: Gelfand duality}

For basic facts about lattice-ordered rings and algebras we use Birkhoff's book \cite[Ch.~XIII and onwards]{Bir79} as our main reference. All rings we consider are assumed to be commutative and unital.

\begin{definition}
\begin{enumerate}
\item[]
\item A ring $A$ with a partial order $\le$ is a \emph{lattice-ordered ring}, or an \emph{$\ell$-ring} for short, provided
$(A,\le)$ is a lattice, $a\le b$ implies $a+c \le b+c$ for each $c$, and $0 \leq a, b$ implies $0 \le ab$.
\item An $\ell$-ring $A$ is an \emph{$\ell$-algebra} if it is an $\mathbb R$-algebra and for each $0 \le a\in A$ and $0\le r\in\mathbb R$ we have $0 \le r\cdot a$.
\item An $\ell$-ring $A$ is \emph{bounded} if for each $a \in A$ there is $n \in \mathbb{N}$ such that $a \le n\cdot 1$ (that is, $1$ is a \emph{strong order unit}).
\item An $\ell$-ring $A$ is \emph{archimedean} if for each $a,b \in A$, whenever $n\cdot a \le b$ for each $n \in \mathbb{N}$, then $a \le 0$.
\item An \emph{$\ell$-algebra morphism} $\alpha : A \to B$ is both an $\mathbb{R}$-algebra and lattice homomorphism. It is \emph{unital} if $\alpha(1) = 1$.
\item Let $\bal$ be the category of bounded archimedean $\ell$-algebras and unital $\ell$-algebra morphisms.
\end{enumerate}
\end{definition}

Let $A\in\bal$. For $a\in A$, define the \emph{absolute value} of $a$ by
\[
|a|=a\vee(-a).
\]
If we set the positive and negative parts of $a$ to be $a^+ = a\vee 0$ and $a^- = (-a)^+$, then $a = a^+ - a^-$ and $|a| = a^+ + a^-$.  

The \emph{norm} of $a$ is defined by
\[
||a||=\inf\{r\in\mathbb R \mid |a|\le r\cdot 1\}.
\]
If $X \in \KHaus$, then $C(X) \in \bal$ and the definition of the norm of $f \in C(X)$ coincides with the usual definition
\[
\|f\| = \sup\{ |f(x)| \mid x \in X\}. 
\]

If $\alpha : A \to B$ is a $\bal$-morphism, it is easy to see that $\|\alpha(a)\| \le \|a\|$, $\alpha(a^+) = \alpha(a)^+$, and $\alpha(r) = r$ for all $r \in \mathbb{R}$, where we identify $r$ with $r \cdot 1$.

\begin{definition}
We call $A\in\bal$ \emph{uniformly complete} if its norm is complete. Let $\ubal$ be the full subcategory of $\bal$ consisting of uniformly complete objects of $\bal$.
\end{definition}

\begin{theorem} [{Gelfand duality}] \label{thm: Gelfand}
There is a dual adjunction between $\bal$ and $\KHaus$ which restricts to a dual equivalence between $\ubal$ and $\KHaus$.
\end{theorem}

We briefly describe the functors $\C : \KHaus\to\bal$ and $\Y : \bal\to\KHaus$ establishing the dual adjunction of Theorem~\ref{thm: Gelfand}; for details see \cite[Sec.~3]{BMO13a} and the references therein. For a compact Hausdorff space $X$ let $\C(X) := C(X)$ be the ring of (necessarily bounded) continuous real-valued functions on $X$. For a continuous map $\varphi : X\to Y$ let $\C(\varphi):C(Y)\to C(X)$ be defined by $\C(\varphi)(f)=f\circ\varphi$ for each $f\in C(Y)$. Then $\C:\KHaus\to\bal$ is a well-defined contravariant functor.

For $A \in \bal$, we recall that an ideal $I$ of $A$ is an \emph{$\ell$-ideal} if $|a|\le|b|$ and $b\in I$ imply $a\in I$.
If $A\in\bal$, then we can associate to $A$ a compact Hausdorff space as follows. Let $Y_A$ be the space of maximal $\ell$-ideals of $A$, whose closed sets are exactly sets of the form
\[
Z_\ell(I) = \{M\in Y_A\mid I\subseteq M\},
\]
where $I$ is an $\ell$-ideal of $A$. It follows from the work of Yosida \cite{Yos41} that $Y_A \in \KHaus$. As is customary, we refer to $Y_A$ as the \emph{Yosida space} of $A$ and set $\Y(A)=Y_A$. For a morphism $\alpha$ in $\bal$ we let $\Y(\alpha)=\alpha^{-1}$. Then $\Y:\bal\to\KHaus$ is a well-defined contravariant functor, and the functors $\C$ and $\Y$ yield a dual adjunction between $\bal$ and $\KHaus$.

For $X\in\KHaus$ we have that $\varepsilon_X:X\to Y_{C(X)}$ is a homeomorphism where
\[
\varepsilon_X(x)=\{f\in C(X) \mid f(x)=0\}.
\]

For $A\in\bal$ define $\zeta_A :A\to C(Y_A)$ by $\zeta_A(a)(M)=r$ where $r$ is the unique real number
satisfying $a+M=r+M$. Then $\zeta_A$ is a monomorphism in $\bal$ separating points of $Y_A$. Therefore, by the Stone-Weierstrass theorem, $\zeta_A : A \to C(Y_A)$ is the uniform completion of $A$. Thus, if $A$ is uniformly complete, then $\zeta_A$ is an isomorphism. Consequently, the contravariant adjunction restricts to a dual equivalence between $\ubal$ and $\KHaus$, yielding Gelfand duality. Another consequence of these considerations is the following well-known result.

\begin{proposition}\label{prop: SW}
$\ubal$ is a full reflective subcategory of $\bal$, and the reflector assigns to each $A \in \bal$ its uniform completion $\C\Y(A) = C(Y_A) \in \ubal$.
\end{proposition}

\begin{remark} \label{rem: compositions}
Since $\C$ and $\Y$ form a dual adjunction between $\bal$ and $\KHaus$, the natural transformations $\zeta$ and $\varepsilon$ satisfy $\Y(\zeta_A) \circ \varepsilon_{Y_A} = 1_{Y_A}$ and $\C(\varepsilon_X) \circ \zeta_{C(X)} = 1_{C(X)}$ for each $A \in \bal$ and $X \in \KHaus$ by \cite[Thm.~IV.1.1]{Mac71}. Moreover, since $\varepsilon$ is a natural isomorphism, $\Y(\zeta_A) = \varepsilon_{Y_A}^{-1}$ and $\zeta_{C(X)} = \C(\varepsilon_X)^{-1}$.
\end{remark}

If $A$ is an $\ell$-subalgebra of $B \in \bal$, we say $A$ is \emph{uniformly dense} in $B$ if $A$ is dense in $B$ with respect to the topology induced by the norm on $B$. In the following lemma we collect several facts that will be used subsequently. 

\begin{lemma} \cite[Lem.~2.9]{BMO13a} \label{lem: bal morphism properties}
Let $\alpha : A \to B$ be a $\bal$-morphism.
\begin{enumerate}[$(1)$]
\item $\Y(\alpha)$ is onto iff $\alpha$ is \emph{1-1} iff $\alpha$ is a monomorphism.
\item $\Y(\alpha)$ is \emph{1-1} iff $\alpha[A]$ is uniformly dense in $B$ iff $\alpha$ is an epimorphism.
\item $\Y(\alpha)$ is a homeomorphism iff $\alpha$ is a bimorphism.
\end{enumerate}
\end{lemma}

\subsection{Free objects in $\bal$} \label{subsec: free objects}

By \cite[Thm.~3.2]{BCM20c}, free objects on nonempty sets do not exist in $\bal$. To see this, observe that each $\bal$-morphism $\alpha : A \to B$ satisfies $\|\alpha(a)\| \le \|a\|$ for each $a \in A$. Now suppose that a free object $F \in \bal$ exists on $X \ne \varnothing$. Let $f : X \to F$ be the corresponding map, let $x \in X$, and let $r \in \mathbb{R}$ satisfy $r > \| f(x)\|$. Define $g : X \to \mathbb{R}$ by $g(y) = r$ for each $y \in X$. Then there is a $\bal$-morphism $\alpha : F \to \mathbb{R}$ with $\alpha \circ f = g$. Therefore, $\alpha(f(x)) = g(x) = r$, violating the inequality $\|\alpha(f(x)) \|\le \|f(x)\|$. Taking this into account leads to the following notion. 

\begin{definition}\cite[Def.~3.3]{BCM20c}
\begin{itemize}
\item A {\em weight function} on a set $X$ is a function $w$ from $X$ into the nonnegative real numbers.
\item A \emph{weighted set} is a pair $(X, w)$ where $X$ is a set and $w$ is a weight function on $X$.
\item A \emph{weighted set morphism} $f : (X_1, w_1) \to (X_2, w_2)$ is a function $f : X_1 \to X_2$ satisfying $w_2(f(x)) \le w_1(x)$ for each $x \in X$.
\end{itemize}
\end{definition}

There is a forgetful functor $U$ from $\bal$ to the category of weighted sets that associates to each $A \in \bal$ the weighted set $(A, \|\cdot\|)$. By \cite[Thm.~3.9]{BCM20c}, $U$ has a left adjoint, thus yielding the following theorem.

\begin{theorem} \label{thm: UMP}
Free objects in $\bal$ exist over weighted sets. 
\end{theorem}

\subsection{Modal operators on bounded archimedean $\ell$-algebras} \label{subsec: modal operators}

In \cite{BCM20a} the notion of a modal operator on $A \in \bal$ was introduced, generalizing that of a modal operator on a boolean algebra. The motivating example comes from a continuous relation $R$ on a compact Hausdorff space $X$ (see Definition~\ref{def: continuous relation}). If $R$ is serial (meaning $R[x] \ne \varnothing$ for each $x \in X$), then there is a natural definition of a modal operator $\Box_R$ on $C(X)$, given by $\Box_R (f)(x) = \inf fR[x]$ for each $x \in X$. It is straightforward to see that $\Box_R$ preserves meet, $0,1$, addition by a scalar, and multiplication by a nonnegative scalar. If $R$ is not serial, then $\Box_R$ needs to be redefined since $R[x]$ may be empty. For an arbitrary continuous relation we define $\Box_R$ by
\[
(\Box_Rf)(x) = \left\{ \begin{array}{cc} \inf fR[x] & \textrm{if }R[x] \ne \varnothing \\ 1 & \textrm{if } R[x] = \varnothing. \end{array} \right.
\]
If $R$ is not serial, then $\Box_R 0 \ne 0$, and the properties of $\Box_R$ become more complicated. Looking carefully at those properties, we arrive at the following definition.

\begin{definition}\cite[Def.~3.10]{BCM20a}\label{def:mbal}
\begin{enumerate}
\item Let $A \in \bal$. We say that a unary function $\Box: A \to A$ is a \textit{modal operator} on $A$ provided $\Box$ satisfies the following
axioms for each $a,b \in A$ and $r \in \mathbb{R}$:
\begin{enumerate}
\item[(M1)] $\Box (a \meet b) = \Box a \meet \Box b$.
\item[(M2)] $\Box r = r + (1 - r)\Box 0$.
\item[(M3)] $\Box (a^+)=(\Box a)^+$.
\item[(M4)] $\Box (a+r)=\Box a + \Box r - \Box 0$.
\item[(M5)] $\Box (r a)= (\Box r)(\Box a)$ provided $r\ge 0$.
\end{enumerate}
\item If $\Box$ is a modal operator on $A\in\bal$, then we call the pair $(A, \Box)$ a \textit{modal bounded archimedean $\ell$-algebra}.
\item Let $\mbal$ be the category of modal bounded archimedean $\ell$-algebras and unital $\ell$-algebra homomorphisms preserving $\Box$.
\item Let $\mubal$ be the full subcategory of $\mbal$ consisting of $(A,\Box)$ with $A \in \ubal$.
\end{enumerate}
\end{definition}

\begin{remark} \label{rem: properties of Box}
Let $(A, \Box) \in \mbal$. Axiom (M1) implies that $\Box$ is order preserving. From (M2) we have $\Box 1 = 1$. Finally (M3) shows that if $0 \le a$, then $\Box a = (\Box a)^+$, so $0 \le \Box a$. In particular, $0 \le \Box 0$.
\end{remark}

\begin{definition} \cite[Sec.~2]{BBH15} \label{def: continuous relation}
\begin{enumerate}
\item A binary relation $R$ on a compact Hausdorff space $X$ is \emph{continuous} if:
\begin{enumerate}
\item $R[x]$ is closed for each $x\in X$.
\item $F\subseteq X$ closed implies $R^{-1}[F]$ is closed.
\item $U\subseteq X$ open implies $R^{-1}[U]$ is open.
\end{enumerate}
\item If $X$ is compact Hausdorff and $R$ is a continuous relation on $X$, then we call $(X, R)$ a \emph{compact Hausdorff frame}.
\item A \emph{bounded morphism} (or \emph{p-morphism}) between $(X,R)$ and $(Y,S)$ is a map $f:X\to Y$ satisfying $f(R[x])=S[f(x)]$ for each  $x\in X$ (equivalently, $f^{-1} (S^{-1}[y]) = R^{-1}[f^{-1}(y)]$ for each $y \in Y$).
\item Let $\KHF$ be the category of compact Hausdorff frames and continuous bounded morphisms.
\end{enumerate}
\end{definition}

\begin{theorem} \cite[Thm.~5.3]{BCM20a} \label{thm: mbal}
There is a dual adjunction between $\mbal$ and $\KHF$ which restricts to a dual equivalence between $\mubal$ and $\KHF$.
\end{theorem}

The functors establishing the adjunction of Theorem~\ref{thm: mbal} extend those of Gelfand duality. If $(A, \Box) \in \mbal$, define $R_\Box$ on $Y_A$ by $xR_\Box y$ if $0 \le a \in y$ implies $\Box a \in x$. Then $\Y(A,\Box) := (Y_A, R_\Box) \in \KHF$.  Going the other direction, if $(X, R) \in \KHF$, define $\Box_R$ on $C(X)$ as above.
Then $\C(X,R) := (C(X), \Box_R) \in \mbal$.

\section{The endofunctor $\H : \bal \to \bal$} \label{sec: H}

In this section we define the endofunctor $\H$ on $\bal$. 
Let $A\in\bal$. Following \cite[Def.~3.7]{BCM20c}, we call an $\ell$-ideal $I$ of $A$ \emph{archimedean} if $A/I$ is archimedean (and hence $A/I\in\bal$). Archimedean $\ell$-ideals were studied by Banaschewski in the category of archimedean $f$-rings (see \cite[App.~2]{Ban97} and \cite{Ban05a}). It is easy to see that the intersection of archimedean $\ell$-ideals is archimedean, and hence for each $S \subseteq A$ there is a least archimedean $\ell$-ideal containing $S$. As is standard, we call it the \emph{archimedean $\ell$-ideal generated by} $S$.
 


As we pointed out in Section~\ref{subsec: free objects}, for each $A\in\bal$, the norm on $A$ is a weight function on $A$. Below we will work with a different weight function on $A$.

\begin{definition}
Let $A \in \bal$. Define $w_A$ on $A$ by $w_A(a) = \max\{ \|a\|, 1\}$. 
\end{definition}

It is clear that $(A,w_A)$ is a weighted set. 
We use $w_A$ in order for a modal operator to be a weighted set morphism (see Lemma~\ref{lem: norm of box a}). The next definition is one of the main definitions of the paper and is motivated by the axioms defining a modal operator on $A\in\bal$.


\begin{definition} \label{def: F(A)}
Let $A \in \bal$. 
\begin{enumerate}
\item Let $F(A)$ be the free object in $\bal$ on the weighted set $(A, w_A)$, and let $f_A : A \to F(A)$ be the associated map. We let $I_A$ be the archimedean $\ell$-ideal of $F(A)$ generated by the following elements, where $a,b\in A$ and $r\in\mathbb R$:
\begin{enumerate}
\item $f_A(a \wedge b) - f_A(a)\wedge f_A(b)$;
\item $f_A(r) - r - (1 - r) f_A(0)$;
\item $f_A(a^+) - f_A(a)^+$;
\item $f_A(a + r) - f_A(a) - f_A(r) + f_A(0)$;
\item $f_A(r a) - f_A(r) f_A(a) $ if $0 \le r$.
\end{enumerate}
\item Let $\H(A) = F(A)/I_A$ and $h_A : A \to \H(A)$ be the composition of $f_A$ with the quotient map $\pi : F(A) \to \H(A)$. 
\item For $a \in A$ let $\Box_a = h_A(a)$.
\end{enumerate}
\end{definition}

\begin{remark} \label{rem: Box relations}
The set $\{\Box_a \mid a \in A\}$ generates $\H(A)$ (in $\bal$), and these generators satisfy the following relations:
\begin{enumerate}
\item[(F1)]$\Box_{a \wedge b} = \Box_a \wedge \Box_b$.
\item[(F2)]$\Box_r = r + (1-r)\Box_0$.
\item[(F3)]$\Box_{a^+} = (\Box_a)^+$.
\item[(F4)]$\Box_{a + r} = \Box_a + \Box_r - \Box_0$.
\item[(F5)]$\Box_{r a} = \Box_r \Box_a$ if $0 \le r$.
\end{enumerate}
\end{remark}

\begin{theorem}
$\H$ is a covariant endofunctor on $\bal$.
\end{theorem}

\begin{proof}
Let $\alpha : A \to B$ be a $\bal$-morphism. Then $\alpha : (A, w_A) \to (B, w_B)$ is a weighted set morphism since
\[
w_B(\alpha(a)) = \max\{\|\alpha(a)\|,1\} \le \max \{\|a\|, 1\}  = w_A(a)
\]
for each $a \in A$. Therefore, there is a unique $\bal$-morphism $\tau : F(A) \to F(B)$ making the following diagram commute.
\[
\begin{tikzcd}[column sep = 5pc]
A \arrow[r, "f_A"] \arrow[d, "\alpha"'] & F(A) \arrow[d, "\tau"] \\
B \arrow[r, "f_B"']  & F(B)
\end{tikzcd}
\]
We show that $\tau(I_A) \subseteq I_B$. From this it will follow that there is an induced $\bal$-morphism $\overline{\tau} : \H(A) \to \H(B)$ such that $\overline{\tau} \circ h_A = h_B \circ \alpha$. To see that $\tau(I_A) \subseteq I_B$, it suffices to show that the five sets of generators (a)--(e) of $I_A$ are sent to $I_B$ by $\tau$. Since the arguments are similar, we only give the argument for the generators of type (a).

Let $a,b \in A$. Then
\begin{align*}
\tau(f_A(a \wedge b) - f_A(a) \wedge f_A(b)) &= \tau f_A(a \wedge b) - \tau f_A(a) \wedge \tau f_A(b) \\
&= f_B\alpha(a \wedge b) - f_B\alpha(a) \wedge f_B\alpha(b) \\
&= f_B(\alpha(a) \wedge \alpha(b)) - f_B\alpha(a) \wedge f_B\alpha(b) \\
&\in I_B.
\end{align*}

%
%
%
Therefore, $\tau$ induces a $\bal$-morphism $\overline{\tau} : \H(A) \to \H(B)$. We set $\H(\alpha) = \overline{\tau}$. It follows that $\H(\alpha)$ is the unique $\bal$-morphism that makes the following diagram commute.
\[
\begin{tikzcd}[column sep = 5pc]
A \arrow[r, "h_A"] \arrow[d, "\alpha"'] & \H(A) \arrow[d, "\H(\alpha)"] \\
B \arrow[r, "h_B"']  & \H(B)
\end{tikzcd}
\]
It is clear that $\H$ sends identity morphisms to identity morphisms. If $\alpha : A \to B$ and $\gamma : B \to C$ are $\bal$-morphisms, then
\[
\H(\gamma\circ \alpha) \circ h_A = h_C \circ \gamma \circ \alpha
= \H(\gamma) \circ h_B \circ \alpha 
= \H(\gamma) \circ \H(\alpha) \circ h_A. 
\]
Since $h_A[A]$ generates $\H(A)$, we see that $\H(\gamma \circ \alpha) = \H(\gamma) \circ \H(\alpha)$. Thus, $\H$ is a covariant functor.
\end{proof}

\begin{remark}\label{rem: H on box elements}
From the commutativity $\H(\alpha) \circ h_A = h_B \circ \alpha$ it follows that $\H(\alpha)(\Box_a) = \Box_{\alpha(a)}$ for each $a \in A$. This will be used subsequently.
\end{remark}

\section{$\Alg(\H)$ and $\mbal$} \label{sec: Alg(H) and mbal}

In this section we show that the category $\Alg(\H)$ of algebras for the endofunctor $\H$ is isomorphic to $\mbal$. We start by recalling the definition of algebras for an endofunctor (see, e.g., \cite[Def.~5.37]{AHS06}).

\begin{definition} \label{def: Alg(T)}
Let $\sf{C}$ be a category and $\T:\sf{C} \to \sf{C}$ an endofunctor on $\sf{C}$.
\begin{enumerate}
\item An {\em algebra} for $\T$ is a pair $(A,f)$ where $A$ is an object of $\sf{C}$ and $f:\T(A) \to A$ is a $\sf{C}$-morphism.
\item Let $(A_1,f_1)$ and $(A_2,f_2)$ be two algebras for $\T$. A {\em morphism} between $(A_1,f_1)$ and $(A_2,f_2)$ is a $\sf{C}$-morphism $\alpha : A_1 \to A_2$ such that the following square is commutative.
\[
\begin{tikzcd}[column sep = 5pc]
\T(A_1) \arrow[d, "f_1"'] \arrow[r, "\T(\alpha)"] &  \T(A_2) \arrow[d, "f_2"] \\
A_1 \arrow[r, "\alpha"'] & A_2
\end{tikzcd}
\]
\item Let $\Alg(\T)$ be the category whose objects are algebras for $\T$ and whose morphisms are morphisms of algebras.
\end{enumerate}
\end{definition}


\begin{lemma} \label{lem: norm of box a}
If $(A, \Box) \in \mbal$, then $\Box : (A, w_A) \to (A, \|\cdot\|)$ is a weighted set morphism.
\end{lemma}

\begin{proof}
Let $0 \le r \in \mathbb{R}$. We first show that $\Box r \le \max\{r,1 \}$.  If $r \le 1$, then $\Box r \le \Box 1 = 1$ by Remark~\ref{rem: properties of Box}. If $1 \le r$, then $\Box r = r + (1-r)\Box 0 \le r$ since $0 \le \Box 0$, again by Remark~\ref{rem: properties of Box}. Therefore, $\Box r \le \max\{r,1\}$.

We next show that $-\Box r \le \Box(-r)$. We have $\Box 0 = \Box (-r + r) = \Box(-r) + \Box r  -\Box 0$, so $0 \le 2\Box 0 = \Box(-r) + \Box r$. Thus, $-\Box r \le \Box(-r)$. 

To finish the proof, let $r = \|a\|$. Then $-r \le a \le r$, so $\Box(-r) \le \Box a \le \Box r$. We have $\Box r \le \max\{ r, 1\}$ and $-\Box r \le \Box(-r)$. Therefore,
\[
-\max\{ \|a\|,1 \} = -\max\{r,1\} \le -\Box r \le \Box(-r) \le \Box a \le \Box r \le \max\{r,1\} = \max\{ \|a\|,1 \}, 
\]
which implies that $\|\Box a\| \le \max\{\|a\|, 1\} = w_A(a)$. Thus, $\Box : (A, w_A) \to (A, \|\cdot\|)$ is a weighted set morphism.
\end{proof}

\begin{lemma} \label{lem: functor M}
There is a covariant functor $\M : \Alg(\H) \to \mbal$ sending $(A, \sigma)$ to $(A, \Box_\sigma)$, where $\Box_\sigma a = \sigma(\Box_a)$ for each $a \in A$, and an $\Alg(\H)$-morphism $\alpha$ to itself.
\end{lemma}

\begin{proof}
Let $(A, \sigma) \in \Alg(\H)$ and define $\Box_\sigma$ on $A$ by $\Box_\sigma a = \sigma(\Box_a)$. It follows from Definition~\ref{def:mbal} and Remark~\ref{rem: Box relations} that $(A, \Box_\sigma) \in \mbal$. If $\alpha : (A, 
\sigma) \to (A', \sigma')$ is an $\Alg(\H)$-morphism,
\[
\begin{tikzcd}[column sep = 5pc]
\H(A) \arrow[r, "\sigma"] \arrow[d,"\H(\alpha)"'] & A \arrow[d, "\alpha"] \\
\H(A') \arrow[r, "\sigma'"'] & A'
\end{tikzcd}
\]
then
\[
\alpha(\Box_\sigma a) = \alpha\sigma(\Box_a) = \sigma'\H(\alpha)(\Box_a) = \sigma'(\Box_{\alpha(a)}) = \Box_{\sigma'} \alpha(a),
\]
where the second-to-last equality follows from Remark~\ref{rem: H on box elements}. Therefore, $\alpha$ is an $\mbal$-morphism. It is clear that $\M$ preserves identity morphisms and compositions. Thus, $\M$ is a covariant functor.
\end{proof}

\begin{lemma} \label{lem: functor N}
There is a covariant functor $\N : \mbal \to \Alg(\H)$ sending $(A, \Box)$ to $(A, \sigma_\Box)$, where $\sigma_\Box(\Box_a) = \Box a$ for each $a \in A$, and an $\mbal$-morphism $\alpha$ to itself.
\end{lemma}

\begin{proof}
Since $\Box$ is a weighted set morphism by Lemma~\ref{lem: norm of box a}, there is a $\bal$-morphism $\tau : F(A) \to A$ satisfying $\tau f_A(a) = \Box a$ by Theorem~\ref{thm: UMP}. It is clear from Definitions~\ref{def:mbal}(1) and~\ref{def: F(A)}(1) that $I_A \subseteq \ker(\tau)$, so there is a $\bal$-morphism $\sigma_\Box : \H(A) \to A$ satisfying $\sigma_\Box(\Box_a) = \Box a$. We set $\N(A,\Box) = (A, \sigma_\Box) \in \Alg(\H)$. If $\alpha : (A, \Box) \to (A', \Box')$ is an $\mbal$-morphism, we show that $\alpha$ is an $\Alg(\H)$-morphism. For this we show that the following diagram commutes.
\[
\begin{tikzcd}[column sep = 5pc]
\H(A) \arrow[r, "\sigma_{\Box}"] \arrow[d,"\H(\alpha)"'] & A \arrow[d, "\alpha"] \\
\H(A') \arrow[r, "\sigma_{\Box'}"'] & A'
\end{tikzcd}
\]
By Remark~\ref{rem: H on box elements}, $\H(\alpha)(\Box_a) = \Box_{\alpha(a)}$. Therefore, because $\alpha$ preserves $\Box$, we have $\alpha\sigma_\Box(\Box_a) = \alpha(\Box a) = \Box \alpha(a)$ and $\sigma_{\Box'}\H(\alpha)(\Box_a) = \sigma_{\Box'}(\Box_{\alpha(a)}) = \Box \alpha(a)$. As $\{\Box_a \mid a \in A\}$ generates $\H(A)$, we see that $\alpha \circ \sigma_\Box = \sigma_{\Box'} \circ \H(\alpha)$, so $\alpha$ is an $\Alg(\H)$-morphism. It is clear that $\N$ preserves identity morphisms and compositions. Thus, $\N$ is a covariant functor.
\end{proof}

\begin{theorem} \label{thm: mbal = Alg(H)}
The functors $\M$ and $\N$ yield an isomorphism of categories between $\Alg(\H)$ and $\mbal$.
\end{theorem}

\begin{proof}
Let $(A, \sigma) \in \Alg(\H)$. Then $\M(A, \sigma) = (A, \Box_\sigma)$. Therefore, $\N\M(A,\sigma) = (A, \sigma_{\Box_\sigma})$ where $\sigma_{\Box_\sigma}(\Box_a) = \Box_{\sigma} a = \sigma(\Box_a)$. Thus, $\sigma_{\Box_\sigma} = \sigma$, and so $\N\M = 1_{\Alg(\H)}$. 

Next, let $(A, \Box) \in \mbal$. Then $\N(A, \Box) = (A, \sigma_\Box)$. Therefore, $\M\N(A,\Box) = (A, \Box_{\sigma_\Box})$. But $\Box_{\sigma_\Box}a = \sigma_\Box(\Box_a) = \Box a$ by the definition of $\sigma_\Box$, so $\Box_{\sigma_\Box} = \Box$. Thus, $\M\N = 1_{\mbal}$. Consequently, $\M$ and $\N$ yield an isomorphism between $\Alg(\H)$ and $\mbal$.
\end{proof}

\section{$\H$ and the Vietoris endofunctor} \label{sec : H and Vietoris}

In this section we relate $\H$ to the Vietoris endofunctor $\V:\KHaus\to\KHaus$ by showing that the Yosida space $Y_{\H(A)}$ for $A \in \bal$ is homeomorphic to $\V(Y_A)$.

Let $X \in \KHaus$. We recall that the Vietoris space $\V(X)$ is the set of closed subsets of $X$, topologized as follows. If $U$ is an open subset of $X$, let
\begin{align*}
\Box_U &= \{ F \in \V(X) \mid F \subseteq U\}, \\
\Diamond_U &= \{ F \in \V(X) \mid F \cap U \ne \varnothing\}.
\end{align*}
The Vietoris topology on $\V(X)$ is the topology with the subbasis 
\[
\{ \Box_U \cap \Diamond_V \mid U,V \textrm{ open in }X \}.
\] 
We extend $\V$ to a functor as follows. If $\varphi : X \to Y$ is a continuous function between compact Hausdorff spaces, define $\V(\varphi) : \V(X) \to \V(Y)$ by $\V(\varphi)(F) = \varphi(F)$, the image of $F$ under $\varphi$. It is well known that $\V(\varphi)$ is a well-defined continuous map.

\begin{remark}
The Vietoris space of $X$ is usually defined as the space of nonempty closed subsets of $X$ (see, e.g., \cite[p.~120]{Eng89}). However, we follow \cite[p.~111]{Joh82} in including $\varnothing$ in $\V(X)$. This is necessary for our considerations since the continuous relation $R$ on $X$ may not be serial, and hence there may be $x \in X$ with $R[x]=\varnothing$. Therefore, $\rho_R(x)=\varnothing$, and we need $\varnothing\in\V(X)$ for $\rho_R$ to be well defined. However, in Section~\ref{sec: serial case} we will consider $\V^*(X) = \V(X) \setminus \{\varnothing\}$ and relate it to the full subcategory of $\mbal$ corresponding to those $(X,R)$ where $R$ is a serial relation. This subcategory will be characterized by the identity $\Box 0 = 0$.
\end{remark}

\begin{lemma} \label{lem: g is a weighted set map}
Let $A \in \bal$. Define $g_A : A \to C(\V Y_A)$ by
\[
g_A(a)(F) = \left\{ \begin{array}{ll} \inf \zeta_A(a)(F) & \textrm{if } F \ne \varnothing;\\ 1 & \textrm{if } F = \varnothing.\end{array}\right.
\]
Then $g_A : (A, w_A) \to (C(\V Y_A), \|\cdot\|)$ is a well-defined weighted set morphism.
\end{lemma}

\begin{proof}
To simplify notation we write $g$ for $g_A$. To see that $g$ is well defined it is sufficient to show that $g(a)$ is continuous for each $a \in A$. Let $r,s \in \mathbb{R}$ with $r < s$. We show that
\[
g(a)^{-1}(r, s) = \left\{ \begin{array}{ll} \Box_{\zeta_A(a)^{-1}(r, \infty)} \cap \Diamond_{\zeta_A(a)^{-1}(-\infty, s) }& \textrm{if } 1 \notin (r, s) \\ 
(\Box_{\zeta_A(a)^{-1}(r, \infty) }\cap \Diamond_{\zeta_A(a)^{-1}(-\infty, s)}) \cup \Box_\varnothing & \textrm{if } 1 \in (r, s).\end{array} \right.
\]
Suppose that $1 \notin (r, s)$. Then $g(a)(F) \in (r, s)$ implies that $F \ne \varnothing$. Therefore, since $F$ is compact and hence $\zeta_A(a)$ attains its infimum on $F$, we have
\begin{eqnarray*}
F \in g(a)^{-1}(r, s) & \textrm{ iff } & r < \inf \zeta_A(a)(F) < s \\
& \textrm{ iff } & r < \min \zeta_A(a)(F) < s \\
& \textrm{ iff } & F \in \Box_{\zeta_A(a)^{-1}(r, \infty)} \cap \Diamond_{\zeta_A(a)^{-1}(-\infty, s)}.
\end{eqnarray*}
On the other hand, if $1 \in (r, s)$, then $\varnothing \in g(a)^{-1}(r, s)$. Therefore, since $\Box_\varnothing = \{\varnothing\}$, the calculation above yields the second case. Thus, $g(a)$ is continuous.

It is left to show that $g$ is a weighted set morphism. Let $a \in A$. Then $w_A(a) = \max\{\|a\|,1\}$. Suppose that $\|a\| = r$. Then $-r \le a \le r$. If $F$ is nonempty, then $-r \le \inf \zeta_A(a)(F) \le r$, so $|\inf \zeta_A(a)(F)| \le r$. Also, $g(a)(\varnothing) = 1$.  Therefore,
\begin{align*}
\|g(a)\| &= \sup \{ |g(a)(F)|\mid F \in \V(Y_A)\} = \sup \{ \{|\inf \zeta_A(a)(F)| \mid F \ne \varnothing\}  \cup  \{1\} \} \\
& = \max\{\sup \{ |\inf \zeta_A(a)(F)| \mid F \ne \varnothing\}, 1\} \le \max\{r, 1\}  = w_A(a).
\end{align*}
Thus, $g : (A, w_A) \to (C(\V Y_A), \|\cdot\|)$ is a weighted set morphism.
\end{proof}

\begin{remark} \label{rem: identifications}
In the proof of Lemma~\ref{lem: map from HA to C(VYA)} 
we identify $X$ with $Y_{C(X)}$ via the homeomorphism $\varepsilon_{X} : X \to Y_{C(X)}$ given in Section~\ref{subsec: Gelfand duality}. We also identify $Y_{A}$ with $\hom_{\bal} (A, \mathbb{R})$ as follows. If $M \in Y_{A}$, then it is well known that $A/M \cong \mathbb{R}$ (see, e.g., \cite[Cor.~2.7]{HJ61}), so there is a $\bal$-morphism $A \to \mathbb{R}$ sending $a \in A$ to $r \in \mathbb{R}$ iff $a-r \in M$.
Conversely, $\rho \in \hom_{\bal}(A, \mathbb{R})$ goes to $\ker(\rho) \in Y_A$. 
\end{remark}

\begin{lemma} \label{lem: map from HA to C(VYA)}
There is a $($unique$)$ $\bal$-morphism $\tau_A : F(A) \to C(\V Y_A)$ satisfying $\tau_A \circ f_A = g_A$, the image of $\tau_A$ is uniformly dense in $C(\V Y_A)$, and $\ker(\tau_A)$ contains $I_A$. Therefore, there is a $($unique$)$ $\bal$-morphism $\eta_A : \H(A) \to C(\V Y_A)$ satisfying $\eta_A \circ h_A = g_A$ and whose image is uniformly dense in $C(\V Y_A)$.
\[
\begin{tikzcd}[column sep = 5pc]
& F(A) \arrow[d, "\pi"'] \arrow[dd, bend left = 50, "\tau_A"] \\ 
A \arrow[r, "h_A"] \arrow[dr, "g_A"'] \arrow[ur, "f_A"] & \H(A) \arrow[d, "\eta_A"'] \\
& C(\V Y_A)
\end{tikzcd}
\]
\end{lemma}

\begin{proof}
The existence and uniqueness of $\tau_A$ follows from Lemma~\ref{lem: g is a weighted set map} and Theorem~\ref{thm: UMP}. To show that the image of $\tau_A$ is uniformly dense, by Lemma~\ref{lem: bal morphism properties}(2) it suffices to show that $\Y(\tau_A) : Y_{C(\V Y_A)} \to Y_{F(A)}$ is 1-1. We may identify $Y_{F(A)}$ with $\hom_{\bal}(F(A), \mathbb{R})$ and $Y_{C(\V Y_A)}$ with $\V(Y_A)$ by Remark~\ref{rem: identifications}.
Under these identifications, if $F \in \V Y_A$ we let $\rho_F \in \hom_{\bal}(F(A), \mathbb{R})$ be the corresponding homomorphism. For $a \in A$ and $r \in \mathbb{R}$ we have
\begin{eqnarray*}
\rho_F(f_A(a)) = r &\textrm{iff} & f_A(a) - r \in \Y(\tau_A)(\varepsilon_{\V Y_A}(F)) \\
&\textrm{iff} &f_A(a) - r \in \tau_A^{-1}(\varepsilon_{\V Y_A}(F) )\\
&\textrm{iff} &\tau_Af_A(a) - r \in \varepsilon_{\V Y_A}(F) \\
&\textrm{iff} & \tau_Af_A(a)(F) = r \\
&\textrm{iff} & g_A(a)(F) = r. 
\end{eqnarray*}
Therefore, $\rho_F$ satisfies $\rho_F(f_A(a)) = \inf \zeta_A(a)(F)$ if $F \ne \varnothing$, and $\rho_\varnothing$ is the function sending each $f_A(a)$ to 1.
To see that $\Y(\tau_A)$ is 1-1, suppose that $C \ne D$. If one of $C, D$ is empty, say $C = \varnothing$, then $\rho_C f_A(0) = 1$ and $\rho_D f_A(0) = \inf \zeta_A(0)(D) = 0$ since $D$ is nonempty. Therefore, $\rho_C \ne \rho_D$.  If $C, D \ne \varnothing$, without loss of generality we may assume that $C \not\subseteq D$. Then there is $y \in Y_A$ with $y \in C$ and $y \notin D$. Since $Y_A$ is compact Hausdorff, there is $b \in C(Y_A)$ with $0 \le b \le 1$, $b(D) = \{1\}$ and $b(y) = 0$. Because $\zeta_A[A]$ is uniformly dense in $C(Y_A)$, there is $a \in A$ with $\|b - \zeta_A(a)\| < 1/3$. Therefore, $\inf \zeta_A(a)(D) \ge 2/3$ and $\inf \zeta_A(a)(C) \le 1/3$. This shows that $\rho_C f_A(a) \ne \rho_D f_A(a)$, so $\rho_C \ne \rho_D$. Thus, $\Y(\tau_A)$ is 1-1, and hence the image of $\tau_A : F(A) \to C(\V Y_A)$ is uniformly dense.

To show that $I_A \subseteq \ker(\tau_A)$, it is sufficient to show that $\ker(\tau_A)$ contains all five classes of generators of $I_A$. Because the proof is similar to that of \cite[Lem.~3.8]{BCM20a}, we only demonstrate~(a).

Let $a,b \in A$. We have 
\[
\tau_A(f_A(a \wedge b) - f_A(a) \wedge f_A(b)) = \tau_A f_A(a \wedge b) - \tau_A f_A(a) \wedge \tau_A f_A(b) = g_A({a \wedge b}) - g_A(a) \wedge g_A(b). 
\]
Therefore, we need to prove that $g_A({{a \wedge b}}) = g_A(a) \wedge g_A(b)$. Both sides send $\varnothing$ to $1$. Suppose that $F \in \V(Y_A)$ is nonempty. Then
\begin{align*}
g_A({a\wedge b})(F) &= \inf (\zeta_A(a) \wedge \zeta_A(b))(F) = \min (\zeta_A(a) \wedge \zeta_A(b))(F) \\
&= \min \{ (\zeta_A(a) \wedge \zeta_A(b))(x) \mid x \in F\}\\
&=  \min\{ \min\{\zeta_A(a)(x), \zeta_A(b)(x)\} \mid x \in F\} \\
&= \min \{ \min \zeta_A(a)(F), \min \zeta_A(b)(F) \} \\
&= (g_A(a) \wedge g_A(b))(F).
\end{align*}
Thus,  $g_A({a \wedge b}) = g_A(a) \wedge g_A(b)$.
%
%
%
%
\end{proof}

We next show that $\eta_A$ is 1-1. For this we require a technical result, 
which is an analogue of \cite[Prop.~4.8]{BCM20a}.

\begin{definition}
Let $A \in \bal$.
\begin{enumerate}[$(1)$]
\item If $x \in Y_{\H(A)}$, set $\Box^{-1}x = \{ a \in A \mid \Box_a \in x\}$.
\item If $S \subseteq A$, set $S^+ = \{ s \in S \mid 0 \le s\}$.
\item Define a binary relation $R^\Box \subseteq Y_{\H(A)} \times Y_A$ by setting $x R^\Box y$ if $y^+ \subseteq \Box^{-1}x$ for each $x \in Y_{\H(A)}$ and $y \in Y_A$. 
\end{enumerate}
\end{definition}

\begin{proposition} \label{prop: difficult}
Let $A \in \bal$ and $x \in Y_{\H(A)}$. Then $(\Box^{-1} x)^+ = \bigcup \{ y^+ \mid y \in Y_A, x R^\Box y \}$.
\end{proposition}

\begin{proof}
The proof is the same as that of \cite[Prop.~4.8]{BCM20a} after replacing $\Box a$ with $\Box_a$ and $R_\Box$ with $R^\Box$. 
\end{proof}

\begin{lemma} \label{lem: Box0}
Let $\rho : \H(A) \to \mathbb{R}$ be a $\bal$-morphism.
\begin{enumerate}[$(1)$]
\item $\rho(\Box_0) \in \{0,1\}$.
\item If $\rho(\Box_0) =1$, then $\rho(\Box_a) = 1$ for each $a \in A$.
\end{enumerate}
\end{lemma}

\begin{proof}
(1) If we set $r = 0 = a$ in (F5), we get $\Box_0 \Box_0 = \Box_0$, so $\Box_0$ is an idempotent. Therefore, $\rho(\Box_0) \in \mathbb{R}$ is an idempotent, and hence $\rho(\Box_0) \in \{0,1\}$.
	
(2) Suppose that $\rho(\Box_0) = 1$. By (F5), $\Box_0 \Box_a = \Box_0$ for each $a \in A$. So applying $\rho$ to both sides yields $\rho(\Box_a) = 1$.
\end{proof}

\begin{theorem} \label{thm: psi_A is 1-1}
For $A \in \bal$, the Yosida space of $\H(A)$ is homeomorphic to $\V(Y_A)$.
\end{theorem}

\begin{proof}
The map $\eta_A : \H(A) \to C(\V Y_A)$ induces a continuous map $\Y(\eta_A) : Y_{C(\V Y_A)} \to Y_{\H(A)}$. We identify $Y_{C(\V Y_A)}$ with $\V(Y_A)$ and $Y_{\H(A)}$ with $\hom_{\bal} (\H(A), \mathbb{R})$ as in Remark~\ref{rem: identifications}.  As we saw in the proof of Lemma~\ref{lem: map from HA to C(VYA)}, under these identifications $\Y(\eta_A)(F) := \rho_F$ satisfies $\rho_F(\Box_a) = \inf \zeta_A(a)(F)$ if $F$ is nonempty, and $\rho_F(\Box_a) = 1$ if $F = \varnothing$. By Lemma~\ref{lem: map from HA to C(VYA)}, the image of $\eta_A$ is uniformly dense in $C(\V Y_A)$. Therefore, $\Y(\eta_A)$ is 1-1 by Lemma~\ref{lem: bal morphism properties}(2). 

To show that $\Y(\eta_A)$ is onto, let $\rho : \H(A) \to \mathbb{R}$ be a $\bal$-morphism. If $\rho(\Box_0) = 1$, then $\rho(\Box_a) = 1$ for all $a \in A$ by Lemma~\ref{lem: Box0}(2). Therefore, $\rho$ and $\rho_{\varnothing}$ agree on each $\Box_a$. Since these generate $\H(A)$, we see that $\rho = \rho_{\varnothing}$.  By Lemma~\ref{lem: Box0}(1), we now may assume that $\rho(\Box_0) = 0$. By (F2), $\rho(\Box_r) = r$ for each $r \in \mathbb{R}$. Let
\[
S = \{ (a - \rho(\Box_a))^- \mid a \in A\}
\]
and $F = \{ M \in Y_A \mid S \subseteq M\}$, a closed subset of $Y_A$. We claim that $\rho = \rho_F$. Let $a \in A$ and $y \in F$. Then $(a - \rho(\Box_a))^- \in y$. This means $0 \le (\zeta_A(a) - \rho(\Box_a))(y)$ by \cite[Rem.~2.11]{BMO16}, so $\rho(\Box_a) \le \zeta_A(a)(y)$. Since this is true for all $y \in F$, we see that $\rho(\Box_a) \le \inf \zeta_A(a)(F)$. Thus, it suffices to prove that for each $a \in A$ there is $y \in F$ with $\zeta_A(a)(y) = \rho(\Box_a)$. In other words, we need to show that there is $y \in F$ with $a - \rho(\Box_a)  \in y$.

Let $x = \ker(\rho) \in Y_{\H(A)}$. If $a \in A$, then
\[
\rho(\Box_{a - \rho(\Box_a)}) = \rho(\Box_a + \Box_{-\rho(\Box_a)} - \Box_0) =  \rho(\Box_a)  - \rho(\Box_a) = 0
\]
by (F4) and the fact that $\rho(\Box_r) = r$. From this and (F3) we see that
\[
\rho(\Box_{(a - \rho(\Box_a))^+}) = \rho(\Box_{a - \rho(\Box_a)}^+) = \rho(\Box_{a - \rho(\Box_a)})^+ =  \max\{\rho(\Box_{a - \rho(\Box_a)}), 0\} = \max\{0,0\} = 0,
\]
which implies that $(a - \rho(\Box_a))^+ \in \Box^{-1} x$. By Proposition~\ref{prop: difficult}, there is $y \in Y_A$ with $xR^\Box y$ and $(a - \rho(\Box_a))^+ \in y$. We show that these two facts imply that $y \in F$ and $\rho(\Box_a) = \zeta_A(a)(y)$.
Let $b \in A$. Since $A/y \cong \mathbb{R}$, there is $r \in \mathbb{R}$ with $b - r \in y$. Therefore, $(b - r)^+ \in y$, so $\Box_{(b - r)^+} \in x$. Because $x = \ker(\rho)$,
\[
0 = \rho(\Box_{(b - r)^+}) = \rho(\Box_{b - r}^+) = \rho(\Box_{b-r})^+ = \max\{\rho(\Box_{b - r}), 0\} = \max\{\rho(\Box_b) - r, 0\},
\]
so $\rho(\Box_b) \le r$. Consequently, $b + y = r + y \ge \rho(\Box_b) + y$, and hence $b - \rho(\Box_b) + y \ge 0 + y$. This implies that $(b - \rho(\Box_b))^- \in y$. Since this is true for all $b \in A$, we get $S \subseteq y$, so $y \in F$.  Moreover, for $b = a$ we have $(a - \rho(\Box_a))^+, (a - \rho(\Box_a))^- \in y$, so $a - \rho(\Box_a) \in y$. By the above, this shows that $\rho = \rho_F$, so $\Y(\eta_A)$ is onto. Thus, $\Y(\eta_A)$ is a homeomorphism.
\end{proof}

\begin{remark}
By Theorem~\ref{thm: psi_A is 1-1}, $Y_{\H(A)}$ is homeomorphic to $\V(Y_A)$. Under this homeomorphism, $R^\Box \subseteq Y_{\H(A)} \times Y_A$ is identified with the relation $R \subseteq \V(Y_A) \times Y_A$ given by $F R y$ iff $y \in F$. From this it follows that $R[F] = F$, and for $U\subseteq Y_A$ open, we have $R^{-1}[U] = \Diamond_U$ and $R^{-1}[Y_A \setminus U] = \V(Y_A) \setminus \Box_U$. Consequently, $R$ is a continuous relation, and hence so is $R^\Box$. 
\end{remark}

\section{$\Alg(\H)$ and $\Coalg(\V)$} \label{sec: Alg(H) and Coalg(V)}

In this section we lift the dual adjunction between $\bal$ and $\KHaus$ to a dual adjunction between $\Alg(\H)$ and $\Coalg(\V)$. We show that this dual adjunction restricts to a dual equivalence between the reflective subcategory $\Algu(\H)$ of $\Alg(\H)$ and $\Coalg(\V)$. The category $\Algu(\H)$ consists of those $(A, \alpha) \in \Alg(\H)$ where $A\in\ubal$. This dual equivalence lifts Gelfand duality. We conclude the section by giving an alternate description of $\Algu(\H)$ as $\Alg(\Hu)$ where $\Hu$ is the endofunctor $\C\Y\H:\ubal\to\ubal$. 
\[
\begin{tikzcd}
\ubal \arrow[r, "\H"] & \bal \arrow[r, "\Y"] & \KHaus \arrow[r, "\C"] & \ubal
\end{tikzcd}
\]

We start by recalling the definition of coalgebras (see, e.g., \cite[Def.~9.1]{Ven07}), which is dual to the definition of algebras for an endofunctor.

\begin{definition} \label{def: Coalg(T)}
\begin{enumerate}
\item[]
\item A {\em coalgebra} for an endofunctor $\T:\sf{C} \to \sf{C}$ is a pair $(B,g)$ where $B$ is an object of $\sf{C}$ and $g:B \to \T(B)$ is a $\sf{C}$-morphism.
\item A {\em morphism} between two coalgebras $(B_1, g_1)$ and $(B_2,g_2)$ for $\T$ is a $\sf{C}$-morphism $\alpha : B_1 \to B_2$ such that the following square is commutative.
\[
\begin{tikzcd}[column sep = 5pc]
B_1 \arrow[d, "g_1"'] \arrow[r, "\alpha"] &  B_2 \arrow[d, "g_2"] \\
\T(B_1) \arrow[r, "\T(\alpha)"'] &\T(B_2)
\end{tikzcd}
\]
\item Let $\Coalg(\T)$ be the category whose objects are coalgebras for $\T$ and whose morphisms are morphisms of coalgebras.
\end{enumerate}
\end{definition}

\begin{lemma} \label{lem: commutativity}
Let $\gamma : A \to A'$ be a $\bal$-morphism. Then the following diagram is commutative.
\[
\begin{tikzcd}[column sep = 5pc]
A \arrow[r, "h_A"'] \arrow[d, "\gamma"'] \arrow[rr, bend left = 20, "g_A"] & \H(A) \arrow[r, "\eta_A"'] \arrow[d, "\H(\gamma)"'] & C(\V Y_A) \arrow[d, "\C\V\Y(\gamma)"] \\
A' \arrow[r, "h_{A'}"] \arrow[rr, bend right = 20, "g_{A'}"'] & \H(A') \arrow[r, "\eta_{A'}"] & C(\V Y_{A'})
\end{tikzcd}
\]
\end{lemma}


\begin{proof}
By Remark~\ref{rem: H on box elements}, $\H(\gamma)(h_A(a)) = \H(\gamma)(\Box_a) = \Box_{\gamma(a)} = h_{A'}\gamma(a)$ for each $a \in A$. This shows that the left square of the diagram is commutative. By definition, $g_A = \eta_A \circ h_A$ and $g_{A'} = \eta_{A'} \circ h_{A'}$. We next show that the outside square is commutative, from which we then derive that the right square is commutative. Let $a \in A$ and $F \in \V(Y_{A'})$. If $F = \varnothing$, then
\[
\C\V\Y(\gamma)(g_A(a))(\varnothing) = g_A(a)(\Y(\gamma)(\varnothing))=g_A(a)(\varnothing)=1=g_{A'} \gamma(a)(\varnothing).
\]
If $F \ne \varnothing$, then naturality of $\zeta$ yields
\begin{align*}
\C\V\Y(\gamma)(g_A(a))(F) &= g_A(a)(\Y(\gamma)(F))  = \inf (\zeta_A(a)\Y(\gamma))(F) \\
& = \inf (\C\Y(\gamma)\circ \zeta_A)(a)(F) = \inf \zeta_{A'}(\gamma(a))(F) \\
&= g_{A'} \gamma(a)(F).
\end{align*}
Thus, $\C\V\Y(\gamma) \circ g_A=g_{A'} \circ \gamma$.
Finally, to see that the right square is commutative, 
\begin{align*}
\C\V\Y(\gamma) \circ \eta_A \circ h_A = \C\V\Y(\gamma) \circ g_A = g_{A'} \circ \gamma=  \eta_{A'} \circ \H(\gamma) \circ h_A.
\end{align*}
This yields $\C\V\Y(\gamma) \circ \eta_A=\eta_{A'} \circ \H(\gamma)$ because the image of $h_A$ generates $\H(A)$. 
\end{proof}

\begin{proposition} \label{prop: F}
There is a contravariant functor $\F: \Alg(\H) \to \Coalg(\V)$.
\end{proposition}

\begin{proof}
By the proof of Theorem~\ref{thm: psi_A is 1-1}, if $A \in \bal$, then $\Y(\eta_A)$ is a homeomorphism. For $(A, \alpha) \in \Alg(\H)$, we set $\F(A,\alpha) = (Y_A, \F_\alpha) \in \Coalg(\V)$, where
\[
\F_\alpha = \varepsilon_{\V(Y_A)}^{-1} \circ \Y(\eta_A)^{-1} \circ \Y(\alpha) : Y_A \to \V(Y_A),
\]
\[
\begin{tikzcd}[column sep = 5pc]
Y_A \arrow[r, "\Y(\alpha)"] \arrow[rrr, bend right = 10, "\F_\alpha"'] &  Y_{\H(A)} \arrow[r, "\Y(\eta_A)^{-1} "]  & Y_{C(\V Y_A)} \arrow[r, "\varepsilon_{\V(Y_A)}^{-1}"] & \V(Y_A)
\end{tikzcd}
\]
If $\gamma : (A, \alpha) \to (A', \alpha')$ is an $\Alg(\H)$-morphism
\[
\begin{tikzcd}[column sep = 5pc]
\H(A) \arrow[r, "\alpha"] \arrow[d, "\H(\gamma)"'] & A \arrow[d, "\gamma"] \\
\H(A') \arrow[r, "\alpha'"'] & A'
\end{tikzcd}
\]
then $\Y(\gamma) : Y_{A'} \to Y_A$ is a continuous map. We define $\F(\gamma) = \Y(\gamma)$. To see that $\Y(\gamma)$ is a $\Coalg(\V)$-morphism, we show that the following diagram is commutative.
\begin{equation} \label{eqn: F}
\begin{tikzcd}[column sep = 5pc]
Y_{A'} \arrow[r, "\F_{\alpha'}"] \arrow[d, "\Y(\gamma)"'] & \V(Y_{A'}) \arrow[d, "\V\Y(\gamma)"] \\
Y_A \arrow[r, "\F_{\alpha}"'] & \V(Y_A)
\end{tikzcd}
\end{equation}
To see this we first show that the following diagram is commutative.
\begin{equation} \label{eqn: twosquare}
\begin{tikzcd}[column sep = 5pc]
\V(Y_{A'}) \arrow[r, "\varepsilon_{\V(Y_{A'})}"] \arrow[d, "\V\Y(\gamma)"'] & Y_{C(\V Y_{A'})} \arrow[r, "\Y(\eta_{A'})"] \arrow[d, "\Y\C\V\Y(\gamma)"] & Y_{\H(A')} \arrow[d, "\Y\H(\gamma)"] \\
\V(Y_{A}) \arrow[r, "\varepsilon_{\V(Y_A)}"'] & Y_{C(\V Y_{A})} \arrow[r, "\Y(\eta_A)"'] & Y_{\H(A)}
\end{tikzcd}
\end{equation}
The left square commutes due to the naturality of $\varepsilon$. For the right square,  $\Y\H(\gamma) \circ \Y(\eta_{A'}) = \Y(\eta_{A'} \circ \H(\gamma))$ and $\Y(\eta_A) \circ \Y\C\V\Y(\gamma) = \Y(\C\V\Y(\gamma) \circ \eta_A)$. These are equal by Lemma~\ref{lem: commutativity}. Now, we show that Diagram~(\ref{eqn: F}) commutes. The equation
\[
 \V\Y(\gamma) \circ \F_{\alpha'} = \F_{\alpha} \circ \Y(\gamma)
 \]
 is equivalent to
\[
\V\Y(\gamma) \circ \varepsilon_{\V(Y_{A'})}^{-1} \circ \Y(\eta_{A'})^{-1} \circ \Y(\alpha') = \varepsilon_{\V(Y_A)}^{-1} \circ \Y(\eta_A)^{-1} \circ \Y(\alpha) \circ \Y(\gamma)
\]
and therefore is equivalent to
\begin{equation} \label{eqn: comm}
\Y(\eta_A) \circ \varepsilon_{\V(Y_A)} \circ \V\Y(\gamma) \circ \varepsilon_{\V(Y_{A'})}^{-1} \circ \Y(\eta_{A'})^{-1} \circ \Y(\alpha') = \Y(\alpha) \circ \Y(\gamma).
\end{equation}
Using the commutativity of Diagram~(\ref{eqn: twosquare}) and Equation~(\ref{eqn: comm}), we see that commutativity of Diagram~(\ref{eqn: F}) is equivalent to the equation
\[
\Y\H(\gamma) \circ \Y(\alpha') = \Y(\alpha) \circ \Y(\gamma).
\]
Since $\gamma$ is an $\Alg(\H)$-morphism, we have $\gamma \circ \alpha = \alpha' \circ \H(\gamma)$. Applying $\Y$ to both sides then yields the commutativity of Diagram~(\ref{eqn: F}). Therefore, $\Y(\gamma)$ is a $\Coalg(\V)$-morphism. It is then straightforward to see that $\F$ is a contravariant functor.
\end{proof}

\begin{proposition} \label{prop: G}
There is a contravariant functor $\G : \Coalg(\V) \to \Alg(\H)$.
\end{proposition}

\begin{proof}
Let $(X, \sigma) \in \Coalg(\V)$. Then $\C(\sigma) : C(\V X) \to C(X)$ is a $\bal$-morphism. We set $\G(X, \sigma) = (X, \G_\sigma)$, where $\G_\sigma = \C(\sigma) \circ \C\V(\varepsilon_X) \circ \eta_{C(X)}$.
\[
\begin{tikzcd}[column sep = 5pc]
\H C(X) \arrow[r, "\eta_{C(X)}"] \arrow[rrr, bend right = 10, "\G_{\sigma}"'] &  C(\V Y_{C(X)}) \arrow[r, "\C\V(\varepsilon_X)"] & C(\V X) \arrow[r, "\C(\sigma)"] & C(X)\end{tikzcd}
\]
If $\varphi : (X, \sigma) \to (X', \sigma')$ is a $\Coalg(\V)$-morphism,
\[
\begin{tikzcd}[column sep = 5pc]
X \arrow[r, "\sigma"] \arrow[d, "\varphi"'] & \V(X) \arrow[d, "\V(\varphi)"] \\
X' \arrow[r, "\sigma'"'] & \V(X')
\end{tikzcd}
\]
We define $\G(\varphi) = \C(\varphi)$. We need to show that $\C(\varphi)$ is an $\Alg(\H)$-morphism.
\[
\begin{tikzcd}[column sep = 5pc]
\H C(X') \arrow[r, "\G_{\sigma'}"] \arrow[d, "\H \C(\varphi)"'] & C(X') \arrow[d, "\C(\varphi)"] \\
\H C(X) \arrow[r, "\G_{\sigma}"'] & C(X)
\end{tikzcd}
\]
We have
\begin{align*}
\C(\varphi) \circ \G_{\sigma'} &= \C(\varphi) \circ \C(\sigma') \circ \C\V(\varepsilon_{X'}) \circ \eta_{C(X')} \\
&= \C(\sigma' \circ \varphi) \circ \C\V(\varepsilon_{X'}) \circ \eta_{C(X')} \\
&= \C(\V(\varphi) \circ \sigma) \circ \C\V(\varepsilon_{X'}) \circ \eta_{C(X')} \\
&= \C(\V(\varepsilon_{X'}) \circ \V(\varphi) \circ \sigma) \circ \eta_{C(X')} \\
&= \C(\V(\varepsilon_{X'} \circ \varphi) \circ \sigma) \circ \eta_{C(X')}
\end{align*}
On the other hand, 
\begin{align*}
\G_{\sigma} \circ \H \C(\varphi) &= \C(\sigma) \circ \C\V(\varepsilon_X) \circ \eta_{C(X)} \circ \H \C(\varphi) \\
&= \C(\sigma) \circ \C\V(\varepsilon_X) \circ \C\V\Y \C(\varphi) \circ \eta_{C(X')} \\
&= \C(\sigma) \circ \C\V(\Y\C(\varphi) \circ \varepsilon_X) \circ \eta_{C(X')}\\
&= \C(\sigma) \circ \C\V(\varepsilon_{X'} \circ \varphi) \circ \eta_{C(X')} \\
&= \C(\V(\varepsilon_{X'} \circ \varphi) \circ \sigma) \circ \eta_{C(X')}
\end{align*}
where the second equality holds by applying Lemma~\ref{lem: commutativity} to $\gamma = \C(\varphi)$ and the fourth equality by the naturality of $\varepsilon$. Thus, $\C(\varphi) \circ \G_{\sigma'} = \G_{\sigma} \circ \H \C(\varphi)$. It is then straightforward to see that $\G$ is a contravariant functor.
\end{proof}

\begin{proposition} \label{prop: xi}
There is a natural isomorphism $\xi : 1_{\Coalg(\V)} \to \F \G$.
\end{proposition}

\begin{proof}
We define $\xi : 1_{\Coalg(\V)} \to \F\G$ as follows. If $(X, \sigma) \in \Coalg(\V)$,  then $\xi_{(X,\sigma)} = \varepsilon_X$.
\begin{equation} \label{eqn: diagram1}
\begin{tikzcd}[column sep = 5pc]
X \arrow[r, "\sigma"] \arrow[d, "\varepsilon_X"'] & \V(X) \arrow[d, "\V(\varepsilon_{X})"] \\
Y_{C(X)} \arrow[r, "\F_{\G_\sigma}"'] & \V(Y_{C(X)})
\end{tikzcd}
\end{equation}
To see that $\varepsilon_X$ is a $\Coalg(\V)$-morphism, we have $\G_{\sigma} = \C(\sigma) \circ \C\V(\varepsilon_X) \circ \eta_{C(X)}$. Therefore,
\begin{align*}
\F_{\G_{\sigma}} &= \varepsilon_{\V(Y_{C(X)})}^{-1} \circ \Y(\eta_{C(X)})^{-1} \circ \Y(\G_{\sigma}) \\
&= \varepsilon_{\V(Y_{C(X)})}^{-1} \circ \Y(\eta_{C(X)})^{-1} \circ \Y(\C(\sigma) \circ \C\V(\varepsilon_X) \circ \eta_{C(X)}) \\ 
&= \varepsilon_{\V(Y_{C(X)})}^{-1} \circ \Y(\eta_{C(X)})^{-1} \circ \Y(\eta_{C(X)}) \circ \Y\C\V(\varepsilon_X) \circ \Y\C(\sigma) \\
&= \varepsilon_{\V(Y_{C(X)})}^{-1} \circ \Y\C\V(\varepsilon_X) \circ \Y\C(\sigma) \\
&= \V(\varepsilon_X) \circ \varepsilon_{\V(X)}^{-1}\circ \Y\C(\sigma) \\
&= \V(\varepsilon_X) \circ \sigma \circ \varepsilon_{X}^{-1}
\end{align*}
where the last two equalities hold since $\varepsilon$ is a natural isomorphism. Composing both sides on the right by $\varepsilon_X$ shows that Diagram~(\ref{eqn: diagram1}) commutes. Thus, $\varepsilon_X$ is a $\Coalg(\V)$-morphism.

To see that $\xi : 1_{\Coalg(\V)} \to \F\G$ is a natural transformation, let $\varphi : (X, \sigma) \to (X', \sigma')$ be a $\Coalg(\V)$-morphism. The following diagram commutes since $\varepsilon$ is a natural transformation. 
\[
\begin{tikzcd}[column sep = 5pc]
X \ar[r, "\varepsilon_X"]  \arrow[d, "\varphi"'] & Y_{C(X)} \arrow[d, "\Y\C(\varphi)"] \\
X' \arrow[r, "\varepsilon_{X'}"'] & Y_{C(X')}
\end{tikzcd}
\]
Because $\xi_{(X,\sigma)} = \varepsilon_X$ and $\xi_{(X', \sigma')} = \varepsilon_{X'}$, it follows that $\xi$ is natural. It is a natural isomorphism since $\xi_{(X,\sigma)} = \varepsilon_X$ is a homeomorphism for each $(X, \sigma) \in \Coalg(\V)$.
\end{proof}

\begin{proposition} \label{prop: kappa}
There is a natural transformation $\kappa : 1_{\Alg(\H)} \to \G\F$.
\end{proposition}

\begin{proof}
We define $\kappa : 1_{\Alg(\H)} \to \G\F$ as follows. Let $(A, \alpha) \in \Alg(\H)$. We set $\kappa_{(A,\alpha)} = \zeta_A$.
\begin{equation}  \label{eqn: diagram2}
\begin{tikzcd}[column sep = 5pc]
\H(A) \arrow[r, "\alpha"] \arrow[d, "\H(\zeta_A)"'] & A \arrow[d, "\zeta_A"] \\
\H C(Y_A) \arrow[r, "\G_{\F_{\alpha}}"'] & C(Y_A)
\end{tikzcd}
\end{equation}
To see that $\zeta_A$ is an $\Alg(\H)$-morphism, we show that Diagram~(\ref{eqn: diagram2}) is commutative. We have $\F_{\alpha} = \varepsilon_{\V(Y_A)}^{-1} \circ \Y(\eta_A)^{-1} \circ \Y(\alpha)$ and so
\begin{align*}
\G_{\F_{\alpha}} &= \C(\F_{\alpha}) \circ \C\V(\varepsilon_{Y_A}) \circ \eta_{C(Y_A)} \\
&= \C(\varepsilon_{\V(Y_A)}^{-1} \circ \Y(\eta_A)^{-1} \circ \Y(\alpha)) \circ \C\V(\varepsilon_{Y_A}) \circ \eta_{C(Y_A)} \\
&= \C\Y(\alpha) \circ \C\Y(\eta_A)^{-1} \circ \C(\varepsilon_{\V(Y_A)})^{-1} \circ \C\V(\varepsilon_{Y_A}) \circ \eta_{C(Y_A)}\\
&= \C\Y(\alpha) \circ \C\Y(\eta_A)^{-1} \circ \zeta_{C\V(Y_A)} \circ \C\V(\varepsilon_{Y_A}) \circ \eta_{C(Y_A)} \\
&= \C\Y(\alpha) \circ \C\Y(\eta_A)^{-1} \circ \zeta_{C\V(Y_A)} \circ \C\V\Y(\zeta_A)^{-1} \circ \eta_{C(Y_A)}
\end{align*}
because $\C(\varepsilon_{\V(Y_A)})^{-1}=\zeta_{C\V(Y_A)}$ and $\varepsilon_{Y_A}=\Y(\zeta_A)^{-1}$ by Remark~\ref{rem: compositions}. Thus, by Lemma~\ref{lem: commutativity} and the naturality of $\zeta$ (used twice),
\begin{align*}
\G_{\F_{\alpha}} \circ \H(\zeta_A) &= \C\Y(\alpha) \circ \C\Y(\eta_A)^{-1} \circ \zeta_{C\V(Y_A)} \circ \C\V\Y(\zeta_A)^{-1} \circ \eta_{C(Y_A)} \circ \H(\zeta_A) \\
&= \C\Y(\alpha) \circ \C\Y(\eta_A)^{-1} \circ \zeta_{C\V(Y_A)} \circ \C\V\Y(\zeta_A)^{-1} \circ \C\V\Y(\zeta_A) \circ \eta_A \\
&= \C\Y(\alpha) \circ \C\Y(\eta_A)^{-1} \circ \zeta_{C\V(Y_A)} \circ \eta_A  \\
&= \C\Y(\alpha) \circ \zeta_{\H(A)} \\
&= \zeta_A \circ \alpha.
\end{align*}
Thus, $\zeta_A \circ \alpha = \G_{\F_{\alpha}} \circ \H(\zeta_A)$, and hence $\zeta_A$ is a $\Coalg(\V)$-morphism.

To show naturality, let $\gamma : (A, \alpha) \to (A', \alpha')$ be an $\Alg(\H)$-morphism. The following diagram commutes since $\zeta$ is a natural transformation.
\[
\begin{tikzcd}[column sep = 5pc]
A \arrow[d,  "\gamma"'] \arrow[r, "\zeta_{A}"] & C(Y_A) \arrow[d, "\C\Y(\gamma)"] \\
A' \arrow[r, "\zeta_{A'}"'] & C(Y_{A'}) 
\end{tikzcd}
\]
Because $\kappa_{(A,\alpha)} = \zeta_A$ and $\kappa_{(A', \alpha')} = \zeta_{A'}$, it follows that $\kappa$ is a natural transformation.
\end{proof}

\begin{theorem} \label{thm: dual adjunction}
The functors $\F$ and $\G$ yield a dual adjunction between $\Alg(\H)$ and $\Coalg(\V)$.
\end{theorem}

\begin{proof}
By \cite[Thm.~IV.1.2]{Mac71} and Propositions~\ref{prop: F}--\ref{prop: kappa}, it suffices to show that
\[
\F(\kappa_{(A,\alpha)}) \circ \xi_{\F(A,\alpha)} = 1_{\F(A,\alpha)}
\]
and
\[
\G(\xi_{(X,\sigma)}) \circ \kappa_{\G(X,\sigma)} = 1_{\G(X,\sigma)}
\]
for each $(A, \alpha) \in \Alg(\H)$ and $(X, \sigma) \in \Coalg(\V)$. We have $\kappa_{(A,\alpha)} = \zeta_A$ and $\xi_{\F(A,\alpha)} = \varepsilon_{Y_A}$.  Since $\F(\kappa_{(A,\alpha)}) = \F(\zeta_A) = \Y(\zeta_A)$ and  $1_{\F(A,\alpha)} = 1_{Y_A}$, the first equation reduces to $\Y(\zeta_A) \circ \varepsilon_{Y_A} = 1_{Y_A}$, which holds by Remark~\ref{rem: compositions}. For the second equation, $\xi_{(X,\sigma)} = \varepsilon_X$ and $\kappa_{\G(X,\sigma)} = \zeta_{C(X)}$. Since $\G(\xi_{(X,\sigma)}) = \G(\varepsilon_X) = \C(\varepsilon_X)$ and  $1_{\G(X,\sigma)} = 1_{C(X)}$, the equation $\G(\xi_{(X,\sigma)}) \circ \kappa_{\G(X,\sigma)} = 1_{\G(X,\sigma)}$ is equivalent to $\C(\varepsilon_X) \circ \zeta_{C(X)} = 1_{C(X)}$, which also holds by Remark~\ref{rem: compositions}. Therefore, $\F$ and $\G$ form a dual adjunction.
\end{proof}

\begin{definition} \label{def: Algu}
Let $\Algu(\H)$ be the full subcategory of $\Alg(\H)$ consisting of those $(A, \alpha)$ with $A \in \ubal$.
\end{definition}

\begin{corollary} \label{cor: dual equivalence}
\begin{enumerate}[$(1)$]
\item[]
\item The functors $\F$ and $\G$ restrict to a dual equivalence between $\Algu(\H)$ and $\Coalg(\V)$.
\item $\Algu(\H)$ is a reflective subcategory of $\Alg(\H)$.
\end{enumerate}
\end{corollary}

\begin{proof}
(1) Let $(A, \alpha) \in \Alg(\H)$. Then $\kappa_{(A,\alpha)} = \zeta_A$ is an isomorphism iff $A \in \ubal$ iff $(A, \alpha) \in \Algu(\H)$. Consequently, $\kappa : 1_{\Algu(\H)} \to \G\F$ is a natural isomorphism by Proposition~\ref{prop: kappa}. Moreover, $\xi$ is a natural isomorphism by Proposition~\ref{prop: xi}. Therefore, $\F$ and $\G$ restrict to a dual equivalence between $\Algu(\H)$ and $\Coalg(\V)$ by \cite[Thm.~IV.4.1]{Mac71}.

(2) By (1), the functors $\F$ and $\G$ form a dual equivalence between $\Algu(\H)$ and $\Coalg(\V)$. If $(A,\alpha) \in \Alg(\H)$, then the morphism $\kappa_{(A,\alpha)}$ is a universal arrow from $(A,\alpha)$ to $\F$ by \cite[Thm.~IV.1.1]{Mac71}. Therefore, $\Algu(\H)$ is a reflective subcatgory of $\Alg(\H)$ (see \cite[p.~89]{Mac71}).
\end{proof}

\begin{proposition}
The functors $\M, \N$ yield an isomorphism between $\Algu(\H)$ and $\mubal$.
\end{proposition}

\begin{proof}
If $(A, \sigma) \in \Algu(\H)$, then $A \in \ubal$, so $\M(A, \sigma) = (A, \Box_\sigma) \in \mubal$. If $(A, \Box) \in \mubal$, then $A \in \ubal$, so $\N(A, \Box) = (A, \sigma_\Box) \in \Algu(\H)$. Therefore, the proof of Theorem~\ref{thm: mbal = Alg(H)} shows that $\M$ and $\N$ restrict to $\Algu(\H)$ and $\mubal$, respectively, to yield an isomorphism.
\end{proof}

We finish this section by giving an alternate view of the category $\Algu(\H)$. 

\begin{definition}
We let $\Hu$ be the endofunctor $\C\Y\H$ on $\ubal$. Therefore, if $A \in \ubal$, then $\Hu(A) = C(Y_{\H(A)})$ and if $\alpha : A \to A'$ is a $\ubal$-morphism, then $\Hu(\alpha) = \C\Y\H(\alpha)$. 
\end{definition}

Recall from Section~\ref{subsec: Gelfand duality} that if $\gamma : A \to B$ is a $\bal$-morphism with $B \in \ubal$, then there is a unique $\bal$-morphism $\gamma^u : C(Y_A) \to B$ with $\gamma^u \circ \zeta_A = \gamma$, where $\gamma^u = \zeta_B^{-1} \circ \C\Y(\gamma)$.
\[
\begin{tikzcd}[column sep = 5pc]
A \arrow[r, "\zeta_A"] \arrow[d, "\gamma"'] & C(Y_A) \arrow[d, "\C\Y(\gamma)"] \arrow[dl, "\gamma^u"] \\
B & C(Y_B) \arrow[l, "\zeta_B^{-1}"] 
\end{tikzcd}
\]

\begin{proposition} \label{prop: Algu(H) versus Alg(Hu)}
There is an isomorphism of categories between $\Algu(\H)$ and $\Alg(\Hu)$.
\end{proposition}

\begin{proof}
We define $\A : \Algu(\H) \to \Alg(\Hu)$ on objects by sending $(A, \alpha)$ to $(A, \alpha^u)$. On morphisms, if $\gamma$ is an $\Alg(\H)$-morphism, then $\A(\gamma) = \gamma$.
\[
\begin{tikzcd}[column sep = 5pc]
\H(A) \arrow[r, "\zeta_{\H(A)}"'] \arrow[rr, bend left = 15, "\alpha"] \arrow[d, "\H(\gamma)"'] & \Hu(A) \arrow[r, "\alpha^u"'] \arrow[d, "\Hu(\gamma)"] & A \arrow[d, "\gamma"] \\
\H(A') \arrow[r, "\zeta_{\H(A')}"]  \arrow[rr, bend right = 15, "\alpha'"'] & \Hu(A') \arrow[r, "(\alpha')^u"] & A'
\end{tikzcd}
\]
To see that $\gamma$ is an $\Alg(\Hu)$-morphism, the left square of the diagram commutes by the naturality of $\zeta$. We have
\[
(\gamma \circ \alpha^u) \circ\zeta_{\H(A)} = \gamma\circ \alpha = \alpha' \circ\H(\gamma) = (\alpha')^u\circ \zeta_{\H(A')} \circ\H(\gamma) = (\alpha')^u \circ \Hu(\gamma)\circ \zeta_{\H(A)}
\]
so $\gamma \circ \alpha^u = (\alpha')^u \circ \Hu(\gamma)$ since $\zeta_{\H(A)}$ is epic. This shows that $\gamma$ is an $\Alg(\Hu)$-morphism. It then follows that $\A$ is a covariant functor.

Going in the opposite direction, we define a functor $\B : \Alg(\Hu) \to \Algu(\H)$ on objects by sending $(A, \alpha)$ to $(A, \alpha \circ \zeta_{\H(A)})$. On morphisms we send a $\Alg(\Hu)$-morphism $\gamma : A \to A'$ to itself. It is clear that $\B$ is a covariant functor.

If $(A, \alpha) \in \Algu(\H)$, then $\A(A, \alpha) = (A, \alpha^u)$, and so $\B\A(A,\alpha) = (A, \alpha^u \circ \zeta_{\H(A)}) = (A, \alpha)$. Therefore, $\B\A = 1_{\Algu(\H)}$. If $(A, \alpha)  \in \Alg(\Hu)$, then $(A, \alpha \circ \zeta_{\H(A)}) \in \Algu(\H)$, and $(\alpha \circ \zeta_{\H(A)})^u = \alpha$. Therefore, $\A\B = 1_{\Alg(\Hu)}$. Thus, $\A,\B$ yield an isomorphism of categories between $\Algu(\H)$ and $\Alg(\Hu)$.
\end{proof}

\section{$\mbal$ and $\KHF$} \label{sec: mbal and KHF}

In this section we show how to derive from our results the dual adjunction between $\mbal$ and $\KHF$ and the dual equivalence between $\mubal$ and $\KHF$ obtained in \cite{BCM20a}.

We start by recalling (see, e.g., \cite[Thm.~2.16]{BBH15}) that there is an isomorphism of categories between $\Coalg(\V)$ and $\KHF$. The isomorphism is determined by the following functors. The functor $\SS : \Coalg(\V) \to \KHF$ sends $(X, \sigma)$ to $(X, R_\sigma) \in \KHF$, where $xR_\sigma y$ if $y \in \sigma(x)$, and $\SS$ sends a $\Coalg(\V)$ morphism to itself. The functor $\T : \KHF \to \Coalg(\V)$ sends $(X,R) \in \KHF$ to $(X, \sigma_R)$, defined by $\sigma_R(x) = R[x]$, and sends a $\KHF$-morphism to itself. 

As a consequence of this and the results of the previous section, we obtain the main result of \cite{BCM20a}.

\begin{theorem} \cite[Thm.~5.3]{BCM20a} \label{thm: mubal = KHF}
There is a dual adjunction between $\mbal$ and $\KHF$ which restricts to a dual equivalence between $\mubal$ and $\KHF$.
\end{theorem}

\begin{proof}
By Theorem~\ref{thm: dual adjunction} the functors $\F$ and $\G$ form a dual adjunction between $\Alg(\H)$ and $\Coalg(\V)$. By Theorem~\ref{thm: mbal = Alg(H)}, the functors $\M,\N$ yield an isomorphism of categories between $\Alg(\H)$ and $\mbal$. The functors $\SS, \T$ yield an isomorphism of categories between $\Coalg(\V)$ and $\KHF$ \cite[Thm.~2.16]{BBH15}. We thus have the following diagram.
\[
\begin{tikzcd}
\mubal \arrow[r, hookrightarrow] & \mbal \arrow[r, shift left = .5ex, "\N"]  & \Alg(\H) \arrow[r, shift left = .5ex, "\F"] \arrow[l, shift left = .5ex, "\M"] & \Coalg(\V) \arrow[r, shift left = .5ex, "\SS"] \arrow[l, shift left = .5ex, "\G"] &  \KHF \arrow[l, shift left = .5ex, "\T"]
\end{tikzcd}
\]
Consequently, $\SS\F\N : \mbal \to \KHF$ and $\M\G\T : \KHF \to \mbal$ yield a dual adjunction which restricts to a dual equivalence between $\mubal$ and $\KHF$.
\end{proof}

\begin{proposition}
$\SS\F\N$ and $\M\G\T$ are precisely the functors $\C$ and $\Y$ yielding the dual adjunction of \cite[Thm.~5.2]{BCM20a}.
\end{proposition}

\begin{proof}
Let $(A, \Box) \in \mbal$. Then $\Y(A,\Box) = (Y_A, R_\Box)$, where we recall from Section~\ref{subsec: modal operators} that $R_\Box$ is defined by $xR_\Box y$ if $y^+ \subseteq \Box^{-1} x$. We have $\N(A, \Box) = (A, \sigma_\Box)$, which satisfies $\sigma_\Box(\Box_a) = \Box a$ for all $a \in A$. Then $\F(A, \sigma_\Box) = (Y_A, \F_{\sigma_\Box})$, where we recall that $\F_{\sigma_\Box} = \varepsilon_{\V(Y_A)}^{-1} \circ \Y(\eta_A)^{-1} \circ \Y(\sigma_\Box)$. Finally, $\SS$ sends this to $(Y_A, R_{\F_{\sigma_\Box}})$, where $xR_{\F_{\sigma_\Box}} y$ if $y \in \F_{\sigma_\Box}(x)$. Let $x \in Y_A$ and $F = \F_{\sigma_\Box}(x) \in \V(Y_A)$. If $M = \varepsilon_{\V(Y_A)}(F) \in Y_{C(\V Y_A)}$, then $M = \{ g \in C(\V Y_A) \mid g(F) = 0\}$ and
\[
\Y(\eta_A)(M) = \eta_A^{-1}(M) = \sigma_\Box^{-1}(x) = \Y(\sigma_\Box)(x).
\]
We show that $R_\Box = R_{\F_{\sigma_\Box}}$. Suppose that $xR_\Box y$, so $\Box y^+ \subseteq x$. To see that $x R_{\F_{\sigma_\Box}} y$, we need to show that $y \in F$. If not, then by Urysohn's lemma and the fact that $\zeta_A[A]$ is uniformly dense in $C(Y_A)$, there is $a \in A$ with $\zeta_A(a)(y) = 0$ and $\zeta_A(a)[F] \ge 1/2$. By replacing $a$ by $a^+$ we may assume that $a \ge 0$. Since $\zeta_A(a)(y) = 0$, we have $a \in y$. Therefore, $\Box a \in x$. This means $\sigma_\Box(\Box_a) \in x$, so $\Box_a \in \sigma_\Box^{-1}(x) = \eta_A^{-1}(M)$. Thus, $\eta_A(\Box_a) \in M$, so $g_A(a) \in M$. Therefore, $\inf g_A(\zeta_A(a))(F) = 0$, which is false by construction of $a$. This shows $y \in F$. 

Conversely, if $xR_{\F_{\sigma_\Box}} y$, then $y \in F$. Let $a \in y^+$. Then $\inf g_A(\zeta_A(a))(F) = 0$ because $a \in y$ and $a \ge 0$. Therefore, $\eta_A(\Box_a) \in M$, so $\Box_a \in \eta_A^{-1}(M) = \sigma_\Box^{-1}(x)$. Thus, $\Box a = \sigma_\Box(\Box_a) \in x$. This shows $\Box y^+ \subseteq x$, so $x R_\Box y$. This completes the proof that $R_{\F_{\sigma_\Box}} = R_\Box$. Therefore, $\Y$ and $\SS\F\N$ agree on the objects of $\mbal$. For morphisms, if $\alpha : (A, \Box) \to (A', \Box')$ is an $\mbal$-morphism, then $\SS\F\N(\alpha) = \SS\F(\alpha) = \SS(\Y(\alpha)) = \Y(\alpha)$. Thus, $\SS\F\N = \Y$.

In the opposite direction, if $(X, R) \in \KHF$, we show that $\C(X,R) = \M\G \T(X,R)$. First, $\C(X, R) = (C(X), \Box_R)$, where we recall from Section~\ref{subsec: modal operators} that $\Box_Rf$ is given by
\[
(\Box_Rf)(x) = \left\{ \begin{array}{cc} \inf fR[x] & \textrm{if }R[x] \ne \varnothing \\ 1 & \textrm{if } R[x] = \varnothing. \end{array} \right.
\]
The functor $\T$ sends $(X,R)$ to $(X, \sigma_R)$, where $\sigma_R(x) = R[x]$. Then $\G$ sends this to $(C(X), \G_{\sigma_R})$, where we recall that $\G_{\sigma_R} = \C(\sigma_R) \circ \C\V(\varepsilon_X) \circ \eta_{C(X)}$. Finally, $(C(X), \G_{\sigma_R})$ is sent by $\M$ to $(C(X), \Box_{\G_{\sigma_R}})$, where $\Box_{\G_{\sigma_R}}f = \G_{\sigma_R}(\Box_f)$. We have
\begin{align*}
\G_{\sigma_R}(\Box_f) &= \C(\sigma_R)(\C\V(\varepsilon_X)(\eta_{C(X)}(\Box_f))) \\
&= \C(\sigma_R)(\C\V(\varepsilon_X)(g_{C(X)}(f))) \\
&= \C(\sigma_R)(g_{C(X)}(f) \circ \V(\varepsilon_X)) \\
&= g_{C(X)}(f) \circ \V(\varepsilon_X) \circ \sigma_R.
\end{align*}
Let $x \in X$. Then $\sigma_R(x) = R[x]$ and $\V(\varepsilon_X)(R[x]) = \varepsilon_X(R[x])$. Therefore, since $f = \zeta_{C(X)}(f) \circ \varepsilon_X$ by Remark~\ref{rem: compositions}, we have
\begin{align*}
g_{C(X)}(f)(\varepsilon_XR[x]) &= \left\{ \begin{array}{cc} \inf \zeta_{C(X)}(f)(\varepsilon_XR[x]) & \textrm{if } R[x] \ne \varnothing \\ 1 & \textrm{if } R[x] = \varnothing \end{array} \right. \\
&= \left\{ \begin{array}{cc} \inf fR[x] & \textrm{if } R[x] \ne \varnothing \\ 1 & \textrm{if } R[x] = \varnothing \end{array} \right. \\
&= (\Box_Rf)(x).
\end{align*}
Thus, $\C$ and $\M\G\T$ agree on objects of $\KHF$. If $\sigma : (X, R) \to (X', R')$ is a $\KHF$-morphism, then $\M\G\T(\sigma) = \M\C(\sigma) = \C(\sigma)$. Consequently, $\M\G\T = \C$.
\end{proof}

We conclude this section with the following diagram showing the relationship between the various categories we have considered, where the curved vertical arrows are reflections.

\[
\begin{tikzcd}[column sep = 5pc]
\mbal \arrow[r, shift left = .5ex, "\N"]  \arrow[d, bend right = 30] & \Alg(\H) \arrow[l, shift left = .5ex, "\M"] \arrow[d, bend right = 30] & \\
\mubal \arrow[u, hookrightarrow] \arrow[r, shift left = .5ex, "\N"] \arrow[d, shift left = .5ex, "\Y"] & \Algu(\H)  \arrow[u, hookrightarrow] \arrow[r, shift left = .5ex, "\A"] \arrow[l, shift left = .5ex, "\M"] \arrow[d, shift left = .5ex, "\F"]  &  \Alg(\Hu) \arrow[dl, bend left = 20, shift left = .5ex, "\F\B"] \arrow[l, shift left = .5ex,  "\B"] \\
\KHF \arrow[r, shift left = .5ex, "\T"] \arrow[u, shift left = .5ex, "\C"] & \Coalg(\V) \arrow[l, shift left = .5ex, "\SS"] \arrow[ur, shift left = .5ex, bend right = 20, "\A\G"] 
\arrow[u, shift left = .5ex, "\G"] & \\
\end{tikzcd}
\]

\section{The serial case} \label{sec: serial case}

We recall that a binary relation $R$ on a set $X$ is serial provided that $R[x]\ne\varnothing$ for each $x\in X$. If $(X,R)\in\KHF$ with $R$ serial, then we can replace $\V(X)$ with $\V^*(X)$ in our considerations. It was shown in \cite[Prop.~7.2]{BCM20a} that the binary relation $R_\Box$ in the dual compact Hausdorff frame of $(A,\Box)\in\mbal$ is serial iff $\Box 0=0$. We thus arrive at the following definition, the notation of which is motivated by modal logic, where the seriality axiom is denoted by $\sf D$ (see, e.g., \cite{CZ97}). 

\begin{definition} \cite[Sec.~7]{BCM20a}
\begin{enumerate}
\item Let $\mbal^{\sf D}$ be the full subcategory of $\mbal$ consisting of those $(A, \Box) \in \mbal$ with $\Box 0 = 0$.
\item Let $\mubal^{\sf D}$ be the full subcategory of $\mbal^{\sf D}$ consisting of those $(A, \Box) \in \mbal^{\sf D}$ with $A\in\ubal$.
\item Let $\KHF^{\sf D}$ be the full subcategory of $\KHF$ consisting of those $(X,R)$ for which $R$ is a serial relation.
\end{enumerate}
\end{definition}

As was pointed out in \cite[Rem.~3.12]{BCM20a}, when $\Box 0=0$, the axioms (M2), (M4), and (M5) simplify to the following axioms:
\begin{enumerate}
\item[(M2$^*$)] $\Box r = r$.
\item[(M4$^*$)] $\Box (a+r)=\Box a + r$.
\item[(M5$^*$)] $\Box (r a)= r(\Box a)$ provided $r\ge 0$.
\end{enumerate}
Moreover, (M2$^*$) follows from (M4$^*$) by setting $a=0$.

In \cite[Sec.~7]{BCM20a} we showed that the functors $\C$ and $\Y$ restrict to yield a dual adjunction between $\mbal^{\sf D}$ and $\KHF^{\sf D}$, which further restricts to a dual equivalence between $\mubal^{\sf D}$ and $\KHF^{\sf D}$. In this section we briefly outline how to derive this result from our considerations by simplifying the definition of $\H$ to produce a functor $\H^*$ such that $\Alg(\H^*)$ is isomorphic to $\mbal^{\sf D}$.  

\begin{definition}
Let $A \in \bal$. 
\begin{enumerate}
\item Let $I^*_A$ be the archimedean $\ell$-ideal of $F(A)$ generated by the following classes of elements.
\begin{description}
\item[(a)] $f_A(a \wedge b) - f_A(a)\wedge f_A(b)$;
\item[(c)] $f_A(a^+) - f_A(a)^+$;
\item[(d${}^*$)] $f_A(a + r)  - f_A(a) - r$;
\item[(e${}^*$)] $f_A(r a) - r f_A(a) $ if $0 \le r$.
\end{description}
\item Let $\H^*(A) = F(A)/I^*_A$ and $h^*_A : A \to \H^*(A)$ be the composition of $f^*_A$ with the quotient map $\pi : F(A) \to \H^*(A)$. 
\item For $a \in A$ let $\Box^*_a = h^*_A(a)$.
\item Let $\V^*$ be the endofunctor on $\KHaus$ that sends $X$ to the subspace $\V^*(X) = \V(X) \setminus \{\varnothing\}$ of $\V(X)$. 
\end{enumerate}
\end{definition}

The table below compares the relations in $\H^*(A)$ to those of $\H(A)$.

\bigskip

\begin{tabular}{|p{.45\linewidth}|p{.45\linewidth}|}
\hline
Relations for $\H^*(A)$ & Relations for $\H(A)$ \\ \hline
\begin{enumerate}
\item $\Box^*_{a \wedge b} = \Box^*_a \wedge \Box^*_b$
\item \quad ---------
\item $\Box^*_{a^+} = (\Box^*_a)^+$
\item $\Box^*_{a + r} = \Box^*_a + r$
\item $\Box^*_{ra} = r\Box^*_a$ if $0 \le r$
\end{enumerate}
&
\begin{enumerate}
\item $\Box_{a \wedge b} = \Box_a \wedge \Box_b$
\item $\Box_r = r + (1-r)\Box_0$
\item $\Box_{a^+} = (\Box_a)^+$
\item $\Box_{a + r} = \Box_a + \Box_r - \Box_0$
\item $\Box_{ra} = \Box_r \Box_a$ if $0 \le r$
\end{enumerate}
\\ \hline
\end{tabular}

\bigskip

\begin{remark} \label{rem: H*}
\begin{enumerate}
\item[]
\item If we set $r = 0$ in (5) we see that $\Box^*_0 = 0$. Furthermore, setting $a = 0$ in (4) yields $\Box^*_r = r$.
\item From the relations above it follows that $\H^*(A)$ is the quotient of $\H(A)$ by the archimedian $\ell$-ideal of $\H(A)$ generated by $\Box_0$. Consequently, if $(A, \alpha) \in \Alg(\H)$ with $\alpha(\Box_0) = 0$, then there is an induced object $(A, \alpha^*) \in \Alg(\H^*)$.
\end{enumerate}
\end{remark}

\begin{theorem}
There is an isomorphism of categories between $\Alg(\H^*)$ and $\mbal^{\sf D}$.
\end{theorem}

\begin{proof}
The functor $\Alg(\H^*) \to \mbal^{\sf D}$ is defined essentially the same as in Lemma~\ref{lem: functor M}. To define the functor in the other direction, if $(A, \Box) \in \mbal^{\sf D}$, then the induced $\bal$-morphism $\sigma_\Box : \H(A) \to A$ satisfies $\sigma_\Box(\Box_0) = 0$, so induces an object $(A, \sigma^*_\Box) \in \Alg(\H^*)$ by Remark~\ref{rem: H*}(2), which satisfies $\sigma^*_\Box(\Box^*_a) = \Box a$ for each $a \in A$. This gives the functor $\mbal^{\sf D} \to \Alg(\H^*)$. The proof that these functors yield an isomorphism is essentially the same as that of Theorem~\ref{thm: mbal = Alg(H)}. 
\end{proof}

In parallel with Definition~\ref{def: Algu}, let $\Algu(\H^*)$ be the full subcategory of $\Alg(\H^*)$ consisting of those $(A, \alpha)$ with $A \in \ubal$. The proof of the following result is similar to that of Theorem~\ref{thm: dual adjunction} and Corollary~\ref{cor: dual equivalence}(1), with small changes similar to those of the previous theorem. We therefore leave out the details.

\begin{theorem} \label{thm: Alg(H*) = Coalg(V*)}
There is a dual adjunction between $\Alg(\H^*)$ and $\Coalg(\V^*)$ which restricts to a dual equivalence between $\Algu(\H^*)$ and $\Coalg(\V^*)$.
\end{theorem}

\begin{theorem}
There is a dual adjunction between $\mbal^{\sf D}$ and $\KHF^{\sf D}$ which restricts to a dual equivalence between $\mubal^{\sf D}$ and $\KHF^{\sf D}$.
\end{theorem}

\begin{proof}
The proof is similar to that of Theorem~\ref{thm: mubal = KHF} but uses Theorem~\ref{thm: Alg(H*) = Coalg(V*)} instead of Theorem~\ref{thm: dual adjunction} and Corollary~\ref{cor: dual equivalence}(1). It also uses the isomorphism between $\Coalg(\V^*)$ and $\KHF^{\sf D}$, which is essentially the same as that between $\Coalg(\V)$ and $\KHF$. To give some detail, if $(X, \sigma) \in \Coalg(\V^*)$, then (by abusing notation) $\SS$ sends it to $(X, R_\sigma)$. Since the image of $\sigma$ is in $\V^*(X)$, the relation $R_\sigma$ is serial, so $(X, R_\sigma) \in \KHF^{\sf D}$. Conversely, if $(X, R) \in \KHF^{\sf D}$, then $R$ is serial, so $\sigma_R(x) = R[x] \ne \varnothing$ for each $x \in X$, so $\sigma_R : X \to \V^*(X)$ is continuous, and hence (again abusing notation) $\T(X,R) = (X, \sigma_R) \in \Coalg(\V^*)$.
\end{proof}

\begin{remark}
A slight change in the argument of Proposition~\ref{prop: Algu(H) versus Alg(Hu)} shows that $\Algu(\H^*)$ is isomorphic to $\Alg(\H^{*u})$, where $\H^{*u}$ is the composition $\C\Y\H^*$.
\end{remark}

\section{Connection to modal algebras and descriptive frames} \label{sec: MA and DL}

In this final section we connect our results with those of Abramsky \cite{Abr88} and Kupke, Kurz, and Venema \cite{KKV04}. We start by recalling the definition of those $A\in\bal$ that are clean as rings (see, e.g., \cite{McG06} and the references therein).

\begin{definition}
We call $A \in \bal$ {\em clean} if each $a \in A$ can be written as $a = e+v$ with $e$ an idempotent and $v$ a unit. Let $\cubal$ be the full subcategory of $\ubal$ consisting of clean rings.
\end{definition}

\begin{lemma} \label{lem: A clean implies H(A) clean}
If $A \in \cubal$, then $\Hu(A) \in \cubal$.
\end{lemma}
\begin{proof}
By \cite[Prop.~5.20]{BMO13a}, if $A \in \cubal$, then $Y_A$ is a Stone space. Therefore, $\V(Y_A)$ is a Stone space, and hence $Y_{\Hu(A)}$ is a Stone space by Theorem~\ref{thm: psi_A is 1-1}. Thus, $\Hu(A) \in \cubal$ by \cite[Prop.~5.20]{BMO13a}.
\end{proof}

To distinguish between $\V$ on $\KHaus$ and $\Stone$, we denote the Vietoris endofunctor on $\Stone$ by $\V^{\sf S}$. By Lemma~\ref{lem: A clean implies H(A) clean}, $\Hu$ restricts to an endofunctor on $\cubal$, which we denote by $\Hc$. The following result is then an immediate consequence of Corollary~\ref{cor: dual equivalence}(1).

\begin{theorem}
There is a dual equivalence between $\Algu(\Hc)$ and $\Coalg(\V^{\sf S})$.
\end{theorem}

We let $\H^\ba$ be the functor of \cite{KKV04} that sends $B \in \ba$ to the free boolean algebra over its underlying meet-semilattice.  It was shown in \cite[Prop.,~3.12]{KKV04} that $\Alg(\H^\ba)$ is isomorphic to the category $\ma$ of modal algebras. In parallel of $\M : \Alg(\H) \to \mbal$ and $\N : \mbal \to \Alg(\H)$, we denote the functors giving the isomorphism by $\M^\ba : \Alg(\H^\ba) \to \ma$ and $\N^\ba : \ma \to \Alg(\H^\ba)$. By \cite[Thm.~6.9]{BCM20a}, the triangle in the diagram below commutes up to natural isomorphism, where $(-)^* : \DFr \to \ma$ and $(-)_* : \ma \to \DFr$ are the functors yielding Esakia-Goldblatt duality, and the functor $\func{Id}$ sends $(A, \Box) \in \mbal$ to $(\func{Id}(A), \Box|_{\func{Id}(A)})$ (see \cite[Lem.~6.5]{BCM20a}). Therefore, there is an equivalence of categories between $\Alg(\Hc)$ and $\Alg(\H^\ba)$, where the functor $\Alg(\Hc) \to \Alg(\H^\ba)$ is the composition $\N^\ba \circ \func{Id} \circ \M$. 

\[
\begin{tikzcd}
\Alg(\Hc) \arrow[d, shift left = .5ex, "\M"] \arrow[rr, "\N^\ba \circ \func{Id} \circ \M"] && \Alg(\H^\ba) \arrow[d, shift left  = .5ex, "\M^{\ba}"] \\
\mcubal \arrow[u, shift left =.5ex, "\N"] \arrow[rr, "\func{Id}"] \arrow[dr, shift left = .5ex, "\Y"] && \ma \arrow[u, shift left = .5ex, "\N^\ba"] \arrow[dl, shift left = .5ex, "(-)_*"]  \\
& \DFr \arrow[ul, shift left = .5ex, "\C"] \arrow[ur, shift left = .5ex, "(-)^*"] &
\end{tikzcd}
\]

\def\cprime{$'$}
\providecommand{\bysame}{\leavevmode\hbox to3em{\hrulefill}\thinspace}
\providecommand{\MR}{\relax\ifhmode\unskip\space\fi MR }
\providecommand{\MRhref}[2]{%
  \href{http://www.ams.org/mathscinet-getitem?mr=#1}{#2}
}
\providecommand{\href}[2]{#2}

\end{document}